\documentclass[11pt]{amsart}
\usepackage{amsmath,amssymb,amsthm,mathtools}
\usepackage[margin=1.1in]{geometry}
\usepackage{graphicx}
\usepackage{appendix}
\usepackage{hyperref}
\hypersetup{colorlinks=true,
            linkcolor=blue,
            anchorcolor=blue,
            citecolor=blue}

\newtheorem{theorem}{Theorem}[section]
\newtheorem{lemma}[theorem]{Lemma}
\newtheorem{proposition}[theorem]{Proposition}
\newtheorem{corollary}[theorem]{Corollary}
\theoremstyle{definition}

\newtheorem*{definition*}{Definition}
\newtheorem{remark}[theorem]{Remark}
\newtheorem{example}[theorem]{Example}

\numberwithin{equation}{section}

\newcommand{\Snp}{\mathbb{H}^{n+1}}

\def\p{\partial}
\def\n{\nabla}

\def\l{\langle}
\let\ringaccent\r
\def\r{\rangle}

\title[Non-homogeneous curvature flows in hyperbolic space]
{Non-homogeneous curvature flows in hyperbolic space}

\author{Hongyi Sheng}
\address{Institute for Theoretical Sciences, Westlake Institute for Advanced Study, Westlake University,
Hangzhou 310030, China}
\email{shenghongyi@westlake.edu.cn}

\author{Weimin Sheng}
\address{School of Mathematical Sciences, Zhejiang University, Hangzhou 310058, China}
\email{weimins@zju.edu.cn}

\author{Jiazhuo Yang}
\address{School of Mathematical Sciences, Zhejiang University, Hangzhou 310058, China}
\email{yangjiazhuo@zju.edu.cn}


\date{}
\begin{document}

\begin{abstract}
Let $\Snp$ be hyperbolic space of sectional curvature $-1$, with a
fixed point $o$.  We study the non-homogeneous curvature flow
$\partial_tX=-f(\rho)\sigma_k^\alpha\nu$ of smooth, closed, strictly
$h$-convex hypersurfaces enclosing $o$, where $\rho$ is the geodesic
distance to $o$ and $\alpha>0$.  The radial weight $f$ is modeled on
$\sinh^\beta r$; we treat both the supercritical regime
$\beta>1+k\alpha$ and the critical regime $\beta=1+k\alpha$.  Under
the structural condition
$\bigl((f^{1/(1+k\alpha)})'/\cosh(\rho/2)\bigr)'\geq0$, the flow
exists smoothly for all time, preserves strict $h$-convexity, and
contracts to $o$.  Normalized in geodesic normal coordinates at $o$,
the radial function converges, exponentially in normalized time, in
$C^\infty(\mathbb S^n)$: to $1$ in the supercritical case, and to a
positive constant $R_\infty$ determined by the initial hypersurface in
the critical case.  No symmetry assumption and no curvature pinching
beyond strict $h$-convexity are required.
\end{abstract}

\maketitle
\tableofcontents
\enlargethispage{4pt}

\section{Introduction}
Curvature flows of closed convex hypersurfaces have been studied
extensively.  In Euclidean space, Huisken~\cite{Hui84} proved that
strictly convex hypersurfaces moving by their mean curvature contract
to round points in finite time.  Chow~\cite{Chow85,Chow87} obtained
the analogous result for the speed $\sigma_n^{1/n}$ and, under an
additional pinching assumption on the initial hypersurface, for the
speed $\sigma_2^{1/2}$.  Andrews~\cite{And94,And07} extended the theory
to general degree-one homogeneous curvature functions that are convex,
or concave and inverse-concave, covering in particular
$\sigma_k^{1/k}$ and $(\sigma_k/\sigma_l)^{1/(k-l)}$,
$1\leq l<k\leq n$.  For speeds that are not homogeneous of degree one,
the asymptotic shape depends essentially on the degree of homogeneity:
for the Gauss curvature flows $\partial_tX=-K^\alpha\nu$, the works of
Tso~\cite{Tso85}, Andrews~\cite{And99}, Andrews--Chen~\cite{AC12},
Guan--Ni~\cite{GN17}, Andrews--Guan--Ni~\cite{AGN16}, and
Brendle--Choi--Daskalopoulos~\cite{BCD17} together show that every
smooth, closed, strictly convex hypersurface contracts to a round
point for all $\alpha>1/(n+2)$, while at the critical exponent
$\alpha=1/(n+2)$, which corresponds to the affine normal flow, the
normalized limit is an ellipsoid~\cite{BCD17}.  For powers of
$\sigma_k$, Li--Wang--Wu~\cite{LWW21} proved, for $n\geq3$, convergence
of the rescaled $\sigma_k^\alpha$-flow to the unit sphere for closed,
strictly convex, axially symmetric hypersurfaces whenever
$\alpha\in[1/k,c(n,k)]$.

In non-Euclidean space forms, new phenomena appear: in hyperbolic
space, ordinary convexity need not be preserved by a general
contracting curvature flow.  Andrews--Chen~\cite{AC} overcame this
difficulty for compact surfaces in $\mathbb H^3$ under positive
intrinsic scalar curvature, and Chen--Huang~\cite{CH} showed that the
$K^\alpha$-flow contracts strictly convex hypersurfaces to a point,
with convergence to a geodesic sphere after rescaling when
$\alpha>1/(n+2)$.  A particularly natural condition for global
hyperbolic flows is horospherical convexity ($h$-convexity): all
principal curvatures are at least $1$, the common principal curvature
of a horosphere.  Flows of $h$-convex hypersurfaces have been studied,
among others, by Cabezas-Rivas--Miquel~\cite{CRM07},
Gerhardt~\cite{Ger11}, Makowski~\cite{Makowski12},
Andrews--Wei~\cite{AW18}, Andrews--Chen--Wei~\cite{ACW21},
Hu--Li--Wei~\cite{HLW22}, Scheuer--Xia~\cite{SX19}, and
Wang--Wei--Zhou~\cite{WWZ23}; Pipoli~\cite{Pipoli22} considered
non-homogeneous expanding flows in hyperbolic space.  In the sphere,
Gerhardt~\cite{Ger15} used polar duality to identify contracting and
expanding flows and to prove round convergence; see also
Guang--Li--Wang~\cite{GLW24} for a flow approach to the Minkowski
problem in the sphere.

A different line of work allows the speed to depend explicitly on the
position of the hypersurface.  In Euclidean space,
Li--Sheng--Wang~\cite{LSWJEMS} introduced the anisotropic Gauss
curvature flow with speed $f(\nu)r^\beta K$ as a parabolic approach to
the Aleksandrov and dual Minkowski problems, and proved
in~\cite{LSW20} that the fully nonlinear flow with speed
$r^\beta\sigma_k$ contracts to the origin and, after rescaling,
converges smoothly to a sphere centred at the origin.
Li--Xu--Zhang~\cite{LXZ} extended the theory to the power-type speed
$r^{\alpha/\beta}\sigma_k^{1/\beta}$ and to star-shaped $k$-convex
hypersurfaces, and Sheng--Yang~\cite{SY} treated the general
non-homogeneous speed $f(r)\sigma_k^\alpha$.  In hyperbolic space,
Hong~\cite{Hong21} considered the normalized flow
\[
  \partial_tX
  =-(\sinh\rho)^{\alpha/\beta}\sigma_k^{1/\beta}\nu+\gamma V,
  \qquad V=\sinh\rho\,\partial_\rho,
\]
with a normalizing constant $\gamma$, and proved, under suitable
parameter and convexity assumptions, its exponential
convergence to a geodesic sphere.

The present paper studies the \emph{unforced} contracting counterpart
of these radial flows in hyperbolic space, and differs
from~\cite{Hong21} in the following essential points.  First, the
flow \eqref{eq:unnormalized-flow} below contains no conformal-field
term: it contracts to the point $o$ in infinite time, and the
normalization is performed at the contraction point.  Since hyperbolic
space has no ambient homothety, this normalization acts on the radial
function in geodesic normal coordinates at $o$, rather than on the
embedding.  Second, we identify the critical order $\beta=1+k\alpha$
of the radial weight: in the supercritical regime $\beta>1+k\alpha$
the normalized radial function converges to $1$, while at the critical
order the limit is a positive constant $R_\infty$ depending on the
initial hypersurface (see Example~\ref{ex:critical-sphere} below).
Third, the structural condition \eqref{eq:main-profile-root-convexity}
allows general radial weights modeled on $\sinh^\beta r$, not only
pure powers.  On the other hand, we work with strictly $h$-convex
initial hypersurfaces, which is a stronger assumption than the
convexity hypotheses of some of the results quoted above.

Let $\mathbb{H}^{n+1}$ be the $(n+1)$-dimensional hyperbolic space of
sectional curvature $-1$, and fix $o\in\mathbb{H}^{n+1}$.  In geodesic
polar coordinates centred at $o$, the hyperbolic metric is
\begin{equation}\label{eq:hyperbolic-metric}
  \overline g=d\rho^2+\sinh^2\rho\,\sigma_{\mathbb S^n},
\end{equation}
where $\rho$ is the geodesic distance from $o$ and
$\sigma_{\mathbb S^n}$ is the standard metric on $\mathbb S^n$.
Throughout this paper, all hypersurfaces are assumed connected.  Let
$X_0:M\to\mathbb H^{n+1}$ be a smooth, closed, connected, strictly
$h$-convex hypersurface enclosing $o$, and let
$f\in C^\infty((0,\infty))\cap C^0([0,\infty))$ satisfy $f(0)=0$ and
$f(r)>0$ for $r>0$.  We consider the non-homogeneous curvature flow
\begin{equation}\label{eq:unnormalized-flow}
  \begin{cases}
    \displaystyle \frac{\partial X}{\partial t}
      =-f(\rho)\sigma_k^\alpha\nu,\\[1mm]
    X(\cdot,0)=X_0,
  \end{cases}
\end{equation}
where $\nu$ is the outward unit normal and $\sigma_k$ is the $k$-th
elementary symmetric function of the principal curvatures.  To the
best of our knowledge, contracting curvature flows in hyperbolic space
of the form \eqref{eq:unnormalized-flow} have not previously been
studied, even for the natural radial weight $f(\rho)=\sinh^\beta\rho$.
The use of $\sinh\rho$, rather than the bare distance $\rho$, reflects
the warped-product geometry of \eqref{eq:hyperbolic-metric} and
provides the reference weight in the structural condition below.

Our results apply to every $\alpha>0$ and to every smooth, closed,
connected, strictly $h$-convex initial hypersurface enclosing $o$: we
impose neither axial symmetry nor any quantitative curvature-pinching
or strictness threshold beyond strict $h$-convexity itself.  We prove
that strict $h$-convexity is preserved, that the solution exists for
all time and contracts to $o$, and that its normalization converges
smoothly and exponentially, in normalized time, to a centred geodesic
sphere.  The proof must incorporate genuinely hyperbolic effects:
besides the normalization issue mentioned above, the
ambient-curvature reaction has the opposite sign from the spherical
case and is controlled through $h$-convexity.

The equation \eqref{eq:unnormalized-flow} is parabolic on the
G\ringaccent{a}rding cone
$\Gamma_k^+=\{\kappa\in\mathbb R^n:\sigma_j(\kappa)>0,\ j=1,\ldots,k\}$.
In this paper we work on the strictly $h$-convex cone
\[
  \Gamma_h^+=\{\kappa\in\mathbb R^n:\kappa_i>1,\ i=1,\ldots,n\}
\]
and use the inverse Weingarten map to obtain a quantitative lower
bound for all normalized principal curvatures.  Strict $h$-convexity
is the hyperbolic analogue of uniform convexity and is stronger than
ordinary geodesic convexity.  By the hyperbolic Hadamard--Stoker
theorem~\cite{Alex77}, a closed, connected, strictly $h$-convex
hypersurface is embedded and bounds a strictly convex body; since that
body contains $o$, the hypersurface is automatically star-shaped with
respect to $o$ and hence admits the radial-graph representation used
below (see Section~\ref{subsec:normalized-flow}).  No separate
star-shapedness hypothesis is therefore needed.

\subsection{The radial-weight conditions}

Throughout the paper $f(0)=0$, $f(r)>0$ for $r>0$, and
$\beta\geq1+k\alpha$.  In the supercritical case
$\beta>1+k\alpha$, we write $f(r)=\sinh^\beta r+g(r)$ and impose the finite
order flatness condition stated in Theorem~\ref{thm:main}.  In the
critical case $\beta=1+k\alpha$, the remainder is required to decay
strictly faster than the model term, as stated in
Theorem~\ref{thm:main-critical}.  In both cases the structural
condition
\[
  \left(f^{\frac1{1+k\alpha}}\right)''
  \geq\frac12\tanh\frac r2
  \left(f^{\frac1{1+k\alpha}}\right)'
\]
(see \eqref{eq:main-profile-root-convexity} and
\eqref{eq:critical-profile-root-convexity} below)
is the radial convexity input in the $h$-convexity and $C^2$
estimates.  By the chain rule, this condition is equivalent to
\[
  \left(\frac{(f^{\frac1{1+k\alpha}})'}{\cosh\frac r2}\right)'
  \geq0
\qquad (r>0),
\]
i.e.\ to the monotonicity stated in the abstract; for
$f(r)=\sinh^\beta r$ with $\beta\geq1+k\alpha$, set
$a:=\beta/(1+k\alpha)\geq1$.  Then
\[
  Q(r):=
  \frac{(f^{\frac1{1+k\alpha}})'}{\cosh\frac r2}
  =a\frac{\sinh^{a-1}r\,\cosh r}{\cosh\frac r2},
\]
and
\[
  \frac{Q'(r)}{Q(r)}
  =(a-1)\coth r+\tanh r-\frac12\tanh\frac r2>0.
\]
Thus the model weight satisfies the structural condition.

For the all-orders estimates in the normalized variables we shall also
use the following scale-invariant regularity condition at the origin:
\begin{equation}\label{eq:profile-symbol-bounds-all-orders}
  \mathcal P_j(R_0):=
  \sup_{0<r\leq R_0}r^{j-\beta}|f^{(j)}(r)|<\infty
  \qquad\text{for every }j=0,1,2,\ldots
  \text{ and }R_0>0.
\end{equation}
Here the derivatives are taken on $(0,\infty)$; no smooth extension of
$f$ across $r=0$ is required.  This condition is used only to make the
higher-order normalized estimates uniform as $\tau\to\infty$.  The
$C^0$, $C^1$, and $C^2$ estimates, preservation of strict
$h$-convexity, and exclusion of finite-time singularities require only
the bounds through $j=2$ that follow from the hypotheses stated above.

\subsection{Main results}

\begin{theorem}\label{thm:main}
Let $1\leq k\leq n$, $\alpha>0$, and suppose $n\geq2$. Let $M_0\subset\mathbb H^{n+1}$
be a smooth, closed, connected, strictly $h$-convex hypersurface
enclosing $o$.

Let
\begin{equation}\label{eq:C0-supercritical-decomposition-hyperbolic}
  \beta>1+k\alpha,
  \qquad
  m:=\lfloor\beta\rfloor,
  \qquad
  g(r):=f(r)-\sinh^\beta r.
\end{equation}
Assume that
\begin{equation}\label{eq:C0-supercritical-flatness-hyperbolic}
  g\in C^{m+1}([0,\infty)),
  \qquad
  g'(0)=\cdots=g^{(m)}(0)=0
\end{equation}
and
\begin{equation}\label{eq:main-profile-root-convexity}
  \left(f^{\frac{1}{1+k\alpha}}\right)''
  \geq\frac12\tanh\frac r2
  \left(f^{\frac{1}{1+k\alpha}}\right)',
  \qquad r\in(0,\infty).
\end{equation}
Assume also the all-orders scale-invariant bounds
\eqref{eq:profile-symbol-bounds-all-orders}.
Then the solution of \eqref{eq:unnormalized-flow} exists smoothly for
all $t\geq0$, remains strictly $h$-convex, and converges to $o$ as
$t\to\infty$.  Moreover, with the normalization $\lambda$ and the
normalized time $\tau$ defined in
\eqref{eq:lambda-definition}--\eqref{eq:time-and-lambda-identities}
of Section~\ref{subsec:normalized-flow}, the normalized radial
function $\widetilde\rho(\cdot,\tau)$ converges to $1$ in
$C^\infty(\mathbb S^n)$ as $\tau\to\infty$, exponentially in
$\tau$: for every $\ell\in\mathbb N$ there are constants
$C_\ell,c_\ell>0$ such that
\[
  \|\widetilde\rho(\cdot,\tau)-1\|_{C^\ell(\mathbb S^n)}
  \leq C_\ell e^{-c_\ell\tau}
  \qquad (\tau\geq0).
\]
\end{theorem}

\begin{theorem}\label{thm:main-critical}
Let $1\leq k\leq n$, $\alpha>0$, and suppose $n\geq2$.  Let
$M_0\subset\mathbb H^{n+1}$ be a smooth, closed, connected, strictly
$h$-convex hypersurface enclosing $o$.  Set
\begin{equation}\label{eq:C0-critical-decomposition-hyperbolic}
  \beta=1+k\alpha,
  \qquad
  g(r):=f(r)-\sinh^{1+k\alpha}r.
\end{equation}
Assume that there exists $\delta>0$ such that
\begin{equation}\label{eq:C0-critical-remainder-hyperbolic}
  g(r)=O\bigl(r^{1+k\alpha+\delta}\bigr)
  \qquad\text{as }r\downarrow0.
\end{equation}
Set
\begin{equation}\label{eq:C0-critical-regularity-order-hyperbolic}
  N:=\left\lceil\beta+\delta\right\rceil
\end{equation}
and assume that $g\in C^N([0,\infty))$.
Assume moreover that
\begin{equation}\label{eq:critical-profile-root-convexity}
  \left(f^{\frac{1}{1+k\alpha}}\right)''
  \geq\frac12\tanh\frac r2
  \left(f^{\frac{1}{1+k\alpha}}\right)',
  \qquad r\in(0,\infty).
\end{equation}
Assume also the all-orders scale-invariant bounds
\eqref{eq:profile-symbol-bounds-all-orders}.
Then the solution of \eqref{eq:unnormalized-flow} exists smoothly for
all $t\geq0$, remains strictly $h$-convex, and converges to $o$ as
$t\to\infty$.  Moreover, with the normalization $\lambda$ and the
normalized time $\tau$ defined in
\eqref{eq:lambda-definition}--\eqref{eq:time-and-lambda-identities}
of Section~\ref{subsec:normalized-flow}, there exists a constant
$R_\infty>0$, depending on $M_0$, such that
$\widetilde\rho(\cdot,\tau)$ converges to $R_\infty$ in
$C^\infty(\mathbb S^n)$ as $\tau\to\infty$, exponentially in
$\tau$: for every $\ell\in\mathbb N$ there are constants
$C_\ell,c_\ell>0$ such that
\[
  \|\widetilde\rho(\cdot,\tau)-R_\infty\|_{C^\ell(\mathbb S^n)}
  \leq C_\ell e^{-c_\ell\tau}
  \qquad (\tau\geq0).
\]
In general, $R_\infty$ depends on the initial hypersurface and need
not equal $1$; see Example~\ref{ex:critical-sphere} below.
\end{theorem}

\begin{remark}[Normalized versus original
time]\label{rem:normalized-versus-original-time}
The convergence rates stated in
Theorems~\ref{thm:main}--\ref{thm:main-critical} are measured in the
normalized time $\tau$ defined in \eqref{eq:normalized-time} of
Section~\ref{subsec:normalized-flow}.  In both regimes one has
$d\lambda/d\tau=\gamma_0\lambda$ with $\gamma_0>0$; since
$\lambda\to\infty$, the normalized time satisfies $\tau\to\infty$ as
$t\to\infty$, and the two times are related by
$d\tau/dt=\lambda^{k\alpha+1-\beta}$, see \eqref{eq:normalized-time}.
In the critical case $\beta=1+k\alpha$, one has $\tau=t$, so the
convergence is exponential also in the original time.  In the
supercritical case, writing $\mu:=\beta-k\alpha-1>0$, one has
\[
  \tau=\frac{\log(1+\mu\gamma_0t)}{\mu\gamma_0},
  \qquad
  e^{-c\tau}=(1+\mu\gamma_0t)^{-c/(\mu\gamma_0)}.
\]
Thus exponential convergence in normalized time corresponds to an
algebraic rate in the original time in the supercritical regime.
\end{remark}

\begin{example}[The critical limit need not equal
$1$]\label{ex:critical-sphere}
Take $k\alpha=1$ and the critical weight $f(r)=\sinh^2r$
(i.e.~$\beta=2=1+k\alpha$).  For a centred geodesic sphere of radius
$r(t)$, one has $\sigma_k^\alpha=\gamma_0\coth r$ with $\gamma_0$
as in \eqref{eq:lambda-definition}, and \eqref{eq:unnormalized-flow}
reduces to
\[
  r'(t)=-\gamma_0\sinh r(t)\cosh r(t),
\]
hence $\tanh r(t)=e^{-\gamma_0t}\tanh r(0)$.  Since
$\lambda(t)=e^{\gamma_0t}$ in the critical case and
$r(t)\sim\tanh r(t)$ as $t\to\infty$, the normalized radial
function satisfies
\[
  \widetilde\rho(t)=\lambda(t)r(t)\longrightarrow\tanh r(0)
  \qquad (t\to\infty),
\]
so $R_\infty=\tanh r(0)<1$.  Thus the limit in
Theorem~\ref{thm:main-critical} genuinely depends on the initial
hypersurface; the same computation with the supercritical weight
$f(r)=\sinh^\beta r$, $\beta>2$, gives $\widetilde\rho(t)\to1$,
consistently with Theorem~\ref{thm:main}.
\end{example}

\begin{remark}[A family of admissible radial weights]
\label{rem:admissible-hyperbolic-profile}
Both the supercritical and critical theorems admit a broad family of
explicit radial weights.  Set
\[
  p:=1+k\alpha,
  \qquad
  a:=\frac{\beta}{p}\geq1,
\]
and let $h\in C^\infty([0,\infty))$ satisfy
\begin{equation}\label{eq:admissible-weight-hypotheses}
  h(0)=h'(0)=0,
  \qquad
  h'(r)\geq0,
  \qquad
  h''(r)\geq0.
\end{equation}
Define
\begin{equation}\label{eq:admissible-hyperbolic-weight-family}
  q_h(r):=e^{p h(r)},
  \qquad
  f_{\beta,h}(r):=\sinh^\beta r\,q_h(r).
\end{equation}
If $w:=f_{\beta,h}^{1/p}=\sinh^a r\,e^{h(r)}$, then $w'>0$ and a
direct computation gives
\begin{align}\label{eq:admissible-weight-structural-check}
  \frac{w''-\tanh r\,w'}{w}
  ={}&a(a-1)\coth^2r+h''(r)\notag\\
     &+\bigl(2a\coth r-\tanh r\bigr)h'(r)
       +h'(r)^2\geq0.
\end{align}
Since $\tanh r>\frac12\tanh(r/2)$ for $r>0$, it follows that
\[
  \left(f_{\beta,h}^{1/p}\right)''
  >\frac12\tanh\frac r2
    \left(f_{\beta,h}^{1/p}\right)',
\]
so the structural condition in
\eqref{eq:main-profile-root-convexity} is satisfied.

Moreover, \eqref{eq:admissible-weight-hypotheses} implies
$q_h(r)=1+O(r^2)$ as $r\downarrow0$.  Therefore, for the decomposition
used in Theorems~\ref{thm:main} and~\ref{thm:main-critical},
\[
  g(r)=f_{\beta,h}(r)-\sinh^\beta r
      =O(r^{\beta+2}).
\]
If $\beta>p$, then $g\in C^{\lfloor\beta\rfloor+1}([0,\infty))$ and
$g^{(j)}(0)=0$ for $1\leq j\leq\lfloor\beta\rfloor$.  If $\beta=p$,
the critical assumptions hold with $\delta=1$.  Standard
differentiation near the origin also shows that
\eqref{eq:profile-symbol-bounds-all-orders} holds for every member of
this family.

For example, taking nonnegative constants $d_0,d_1,d_2$
gives the admissible factors
\[
  q_h(r)=\cosh^{d_0}r
  \exp\!\bigl(d_1r^2+d_2(\cosh r-1)\bigr).
\]
Thus $\sinh^\beta r$, $\sinh^\beta r\cosh^\gamma r$,
$\sinh^\beta r e^{\gamma r^2}$, and
$\sinh^\beta r e^{\gamma(\cosh r-1)}$, with $\gamma\geq0$, are all
covered by the two theorems.
\end{remark}

\subsection{Outline of the proof}

The normalization is performed at the level of the radial function in
normal coordinates at $o$; there is no ambient homothety in hyperbolic
space.  After recording the radial graph geometry and the normalized
evolution equations, we prove uniform $C^0$ and $C^1$ estimates and a
positive lower bound for the normalized speed.  The central step is the
$C^2$ estimate.  We apply the maximum principle, in the barrier sense,
to the largest eigenvalue of the inverse normalized Weingarten map.
Inverse concavity controls the gradient terms, while the supporting-horosphere
inequality and the structural condition on the radial weight preserve
strict
$h$-convexity directly from the initial time, without an auxiliary
initial layer or a prescribed lower bound for the initial $h$-convexity
margin.  The speed upper bound and the complete curvature
bound then follow.  Finally, a maximum-principle argument gives
exponential decay of the normalized gradient; Evans--Krylov theory,
Schauder estimates, and interpolation yield smooth exponential
convergence to a centred geodesic sphere.

\subsection*{Notation}

We write $\Phi=f(\rho)\sigma_k^\alpha(h)$ and
$F=\sigma_k^\alpha$.  The normalized radial function, Weingarten map,
speed, and support function are denoted by $\widetilde\rho$, $\hat h$,
$\widetilde\Phi$, and $\widetilde u$, respectively.  Whenever
$\hat h$ is positive definite, its inverse is denoted by $\check h$.
For a strictly $h$-convex initial hypersurface $M_0$ we write
\[
  \varepsilon_h(M_0)
  :=\min_{x\in M_0,\,1\leq i\leq n}\bigl(\kappa_i(x,0)-1\bigr)>0,
\]
so that $\kappa_i(\cdot,0)\geq1+\varepsilon_h(M_0)$ on $M_0$.
Constants $c,C$ may change from line to line and, unless stated
otherwise, depend only on $n,k,\alpha$, the fixed base point $o$, the
exponent $\beta$, the initial radial bounds, the initial normalized
speed bounds whenever they are used, and the stated bounds for
$f,f'$, and $f''$ on the relevant radial interval.  Higher-order
constants may additionally depend on the corresponding higher norms
of the initial hypersurface and on the quantities $\mathcal P_j(R_0)$
in \eqref{eq:profile-symbol-bounds-all-orders}.  Statements of
independence from $\varepsilon_h(M_0)^{-1}$ are understood uniformly
over families for which all these listed data remain controlled; any
dependence on $\tau_{\max}$ will be stated explicitly.

We also fix a convention for time derivatives of spatial extrema.  If
$u$ is a smooth function on $\mathbb S^n\times I$ for a compact
smooth time interval $I$, then the extremal functions
$u_{\max}(\tau):=\max_{\mathbb S^n}u(\cdot,\tau)$ and
$u_{\min}(\tau):=\min_{\mathbb S^n}u(\cdot,\tau)$ are locally
Lipschitz on $I$, and hence differentiable almost everywhere.  Their
upper Dini derivatives satisfy
\[
  D^+u_{\max}(\tau)
  \leq\max_{\theta\in\operatorname{Argmax}u(\cdot,\tau)}
      \partial_\tau u(\theta,\tau),
  \qquad
  D^+u_{\min}(\tau)
  \geq\min_{\theta\in\operatorname{Argmin}u(\cdot,\tau)}
      \partial_\tau u(\theta,\tau).
\]
Throughout this paper, every differential inequality involving the
time derivative of a spatial maximum or minimum is understood in this
sense---equivalently, at every differentiability time of the extremal
function---and integrates to the corresponding estimate in the usual
way.

\section{Preliminaries}
\subsection{Radial graphs in
\texorpdfstring{$\mathbb{H}^{n+1}$}{hyperbolic space}}
We first record some basic geometric quantities of a radial graph in
$\mathbb H^{n+1}$. Let
$X(x)=(\rho(x),x)$, $x\in\mathbb S^n$, where
$\rho:\mathbb S^n\to(0,\infty)$ is smooth. All derivatives and
contractions below are taken with respect to $\sigma_{\mathbb S^n}$.

The induced metric and its inverse are
\begin{align}
  g_{ij}
  &=\rho_i\rho_j+\sinh^2\rho\,\sigma_{ij},
    \label{eq:induced-metric}\\
  g^{ij}
  &=\frac{1}{\sinh^2\rho}
    \left(
      \sigma^{ij}
      -\frac{\rho^i\rho^j}{\sinh^2\rho\,v^2}
     \right).
     \label{eq:inverse-induced-metric}
\end{align}
The outward unit normal is
\begin{equation}\label{eq:radial-unit-normal}
  \nu=\frac{1}{v}
  \left(
    \partial_\rho-\frac{\rho^i}{\sinh^2\rho}\partial_i
  \right),
\end{equation}
where
\begin{equation}\label{eq:radial-graph-factor}
  v=\left(1+\frac{|\nabla\rho|^2}{\sinh^2\rho}\right)^{1/2}.
\end{equation}
With the convention $h_{ij}=-\langle\vec h_{ij},\nu\rangle$, the second
fundamental form is
\begin{equation}\label{eq:second-fundamental-form-rho}
  h_{ij}=\frac{1}{v}
  \left(
    -\rho_{ij}
    +\sinh\rho\cosh\rho\,\sigma_{ij}
    +2\coth\rho\,\rho_i\rho_j
  \right).
\end{equation}
The Weingarten map $h_i{}^j=g^{jk}h_{ik}$ is
\begin{equation}\label{eq:weingarten-map-rho}
  h_i{}^j
  =\frac{1}{v\sinh^2\rho}
  \left(
    -\rho_i{}^j
    +\sinh\rho\cosh\rho\,\delta_i{}^j
    +\frac{\rho^j\rho^k\rho_{ki}}
           {\sinh^2\rho\,v^2}
    +\frac{\coth\rho}{v^2}\rho_i\rho^j
  \right).
\end{equation}
The hyperbolic support function is
\begin{equation}\label{eq:radial-support-function}
  u=\langle\sinh\rho\,\partial_\rho,\nu\rangle
   =\frac{\sinh\rho}{v}.
\end{equation}

\subsection{The radial support estimate for
\texorpdfstring{$h$}{h}-convex hypersurfaces}
\label{subsec:supporting-horospheres}

For a smooth domain, the global supporting-horoball definition of
$h$-convexity is equivalent to the local condition $\kappa_i\geq1$;
see \cite[p.~1185]{AW18}.  The only consequence needed below is the
following sharp pointwise estimate.

\begin{lemma}[Radial support estimate]
\label{lem:global-supporting-horosphere}
Let $M\subset\mathbb H^{n+1}$ be a smooth, closed, connected, embedded
hypersurface bounding a domain $\Omega$, and let $\nu$ be its outward
unit normal. Suppose that
\[
  \kappa_i\geq1,\qquad i=1,\ldots,n,
\]
and that $o\in\operatorname{int}\Omega$.  Let
\[
  \rho:\mathbb H^{n+1}\longrightarrow[0,\infty),\qquad
  \rho(Y):=d_{\mathbb H^{n+1}}(o,Y),
\]
be the ambient distance function from $o$, and let
$\partial_\rho:=\bar\nabla\rho$ denote its gradient vector field on
$\mathbb H^{n+1}\setminus\{o\}$.  Then, on $M$,
\begin{equation}\label{eq:h-convex-radial-support-estimate}
  \langle\partial_\rho,\nu\rangle>\tanh\frac{\rho}{2}.
\end{equation}
\end{lemma}

\begin{proof}
Fix $X\in M$.  Since $o\in\operatorname{int}\Omega$ and
$X\in\partial\Omega$, we have $\rho=\rho(X)>0$.  Let
$c:[0,\infty)\to\mathbb H^{n+1}$ be the unit-speed geodesic with
$c(0)=X$ and $c'(0)=-\nu$, and consider the geodesic balls
$B_t:=B(c(t),t)$.  They are nested increasing in $t$: if $t_2>t_1$
and $Y\in B_{t_1}$, then the triangle inequality gives
\[
  d(Y,c(t_2))
  \leq d(Y,c(t_1))+d\bigl(c(t_1),c(t_2)\bigr)
  \leq t_1+(t_2-t_1)=t_2.
\]
Each sphere $\partial B_t$ passes through $X$ and is tangent to $M$
there with outward normal $\nu$: the unit-speed geodesic from the
centre $c(t)$ to $X$ is the reversed curve $s\mapsto c(t-s)$, whose
velocity at $X$ is $-c'(0)=\nu$.  Hence
\[
  H_X:=\bigcup_{t>0}B_t
\]
is the open horoball through $X$ with outward normal $\nu$; this is
the intrinsic description of a horoball as the increasing union of
geodesic balls along a geodesic ray.  Since a horoball through a
given point with a given outward normal is unique, $H_X$ is the
supporting horoball at $X$ provided by the supporting-horoball
property \cite[p.~1185]{AW18}, and therefore
\[
  \Omega\subset\overline{H_X}.
\]
Because $o\in\operatorname{int}\Omega$, the point $o$ lies in the
open horoball $H_X$, so there exists $t>0$ such that
\begin{equation}\label{eq:origin-inside-supporting-horoball}
  d\bigl(o,c(t)\bigr)<t.
\end{equation}

Now apply the hyperbolic law of cosines to the triangle with
vertices $o$, $X$, and $c(t)$.  The unit-speed geodesic from $X$ to
$o$ has initial velocity $-\partial_\rho|_X$, and the geodesic from
$X$ to $c(t)$ has initial velocity $-\nu$; the cosine of the angle
at $X$ is therefore $\langle\partial_\rho,\nu\rangle$.  The three
side lengths are $\rho$, $t$, and $d(o,c(t))$, so
\[
  \cosh d\bigl(o,c(t)\bigr)
  =\cosh\rho\,\cosh t
   -\sinh\rho\,\sinh t\,\langle\partial_\rho,\nu\rangle.
\]
Since $d(o,c(t))<t$ and $\cosh$ is strictly increasing on
$[0,\infty)$, it follows that
\[
  \cosh\rho\,\cosh t
  -\sinh\rho\,\sinh t\,\langle\partial_\rho,\nu\rangle
  <\cosh t,
\]
and hence
\[
  \langle\partial_\rho,\nu\rangle
  >\frac{\cosh\rho-1}{\sinh\rho}\,\coth t
  =\tanh\frac{\rho}{2}\,\coth t
  >\tanh\frac{\rho}{2}.
\]
This proves \eqref{eq:h-convex-radial-support-estimate}.
\end{proof}
\subsection{Inverse concavity of
\texorpdfstring{$\sigma_k^{1/k}$}{sigma-k to the power 1/k}}
\label{subsec:sigma-k-root-inverse-concavity}

We record the algebraic form of inverse concavity that will be used in
the curvature estimate; see Li--Sheng--Wang
\cite[Lemma~3.6]{LSW20}.

\begin{lemma}[Inverse concavity of $\sigma_k^{1/k}$]
\label{lem:sigma-k-root-inverse-concavity}
Let $1\leq k\leq n$ and set
\[
  G(A):=\sigma_k^{1/k}(A),
  \qquad A\in\operatorname{Sym}^+(n).
\]
Then the dual function
\[
  G_*(A):=G(A^{-1})^{-1}
  =\left(\frac{\sigma_n(A)}{\sigma_{n-k}(A)}\right)^{1/k},
\]
is concave on $\operatorname{Sym}^+(n)$.  Equivalently, if
$B=(b^{ij})=A^{-1}$, then, for every symmetric matrix
$\eta=(\eta_{ij})$,
\begin{equation}\label{eq:sigma-k-root-inverse-concavity}
  \ddot G^{ij,kl}\eta_{ij}\eta_{kl}
  +2\dot G^{ij}b^{pq}\eta_{ip}\eta_{jq}
  \geq
  2G^{-1}\bigl(\dot G^{ij}\eta_{ij}\bigr)^2
\end{equation}
where the quantities $G$, $\dot G$, and $\ddot G$ are all evaluated at $A$.
\end{lemma}

\subsection{The normalized flow}\label{subsec:normalized-flow}
We first explain why no separate star-shapedness hypothesis is needed.
A smooth, closed, strictly $h$-convex hypersurface in
$\mathbb H^{n+1}$ enclosing $o$ is automatically a radial graph over
$\mathbb S^n$ centred at $o$.  Indeed, the hyperbolic
Hadamard--Stoker theorem~\cite{Alex77} shows that such a hypersurface is
embedded and bounds a strictly convex body
$\Omega\Subset\mathbb H^{n+1}$.  If a geodesic ray from $o$ met $\partial\Omega$ in more than one
point, or tangentially, the totally geodesic supporting hyperplane at
the last intersection point would separate $o$ from $\Omega$, contrary
to $o\in\operatorname{int}\Omega$.  Thus $M_t$ is star-shaped with
respect to $o$ as long as it remains strictly $h$-convex and encloses $o$.
Hence, after a time-dependent tangential
reparametrization, $M_t$ can be written as
\begin{equation}\label{eq:time-dependent-radial-graph}
  X(\theta,t)=\exp_o\bigl(\rho(\theta,t)\theta\bigr),
  \qquad \theta\in\mathbb S^n.
\end{equation}
For this parametrization,
\begin{equation}\label{eq:radial-velocity-normal-component}
  \left\langle\partial_tX,\nu\right\rangle
  =\rho_t\left\langle\partial_\rho,\nu\right\rangle
  =\frac{\rho_t}{v}.
\end{equation}
The tangential reparametrization does not change the normal velocity.
Comparing \eqref{eq:radial-velocity-normal-component} with
\eqref{eq:unnormalized-flow}, the geometric flow is equivalent to the
scalar equation on $\mathbb S^n$
\begin{equation}\label{eq:radial-flow}
  \begin{cases}
    \displaystyle
    \partial_t\rho
    =-f(\rho)\sigma_k^\alpha\bigl(h[\rho]\bigr)v,
       &\text{on }\mathbb S^n\times[0,T_{\max}),\\[1mm]
    \rho(\cdot,0)=\rho_0.
  \end{cases}
\end{equation}
Here $h[\rho]$ is the Weingarten map in
\eqref{eq:weingarten-map-rho}, and $T_{\max}\in(0,\infty]$ denotes
the maximal existence time of the smooth solution.  Equation
\eqref{eq:radial-flow} is
parabolic on the admissible branch $h[\rho]\in\Gamma_k^+$.

Hyperbolic space has no ambient homothety.  Consequently, the
normalization is defined by blowing up the radial function in normal
coordinates at $o$, rather than by multiplying the embedding $X$.
Define
\begin{equation}\label{eq:lambda-definition}
  \lambda(t)=
  \begin{cases}
    e^{\gamma_0 t},&\beta=1+k\alpha,\\[1mm]
    \bigl(1+(\beta-k\alpha-1)\gamma_0 t\bigr)^%
      {1/(\beta-k\alpha-1)},&\beta>1+k\alpha,
  \end{cases}
\end{equation}
where $\gamma_0 = \binom nk^\alpha$ so that
\begin{equation}\label{eq:lambda-ode}
  \lambda'=\gamma_0\lambda^{2+k\alpha-\beta}.
\end{equation}
Introduce the normalized time
\begin{equation}\label{eq:normalized-time}
  \tau=
  \begin{cases}
    t,&\beta=1+k\alpha,\\[1mm]
    \displaystyle
    \frac{\log\bigl(1+(\beta-k\alpha-1)\gamma_0 t\bigr)}
    {(\beta-k\alpha-1)\gamma_0},&\beta>1+k\alpha.
  \end{cases}
\end{equation}
Equivalently,
\begin{equation}\label{eq:time-and-lambda-identities}
  \frac{d\tau}{dt}=\lambda^{1+k\alpha-\beta},
  \qquad
  \lambda=e^{\gamma_0\tau}.
\end{equation}
The normalized radial function and the corresponding auxiliary radial graph are
\begin{equation}\label{eq:normalized-radius}
  \widetilde\rho(\theta,\tau)=\lambda(t(\tau))\rho(\theta,t(\tau)),
  \qquad
  \widetilde M_\tau
  =\bigl\{\exp_o(\widetilde\rho(\theta,\tau)\theta):
      \theta\in\mathbb S^n\bigr\}.
\end{equation}
In what follows, we use only the normalized radial function
$\widetilde\rho$ and do not consider the geometry of the auxiliary graph
$\widetilde M_\tau$, since the curvatures of $\widetilde M_\tau$ and
$M_{t(\tau)}$ are not related by scaling.  

Differentiating $\widetilde\rho=\lambda\rho$ with respect to $t$ and
using \eqref{eq:radial-flow}, we obtain
\begin{equation}\label{eq:t-derivative-before-curvature-scaling}
  \begin{aligned}
  \partial_t\widetilde\rho
  &=(\partial_t\lambda)\rho+\lambda\partial_t\rho\\
  &=\gamma_0\lambda^{2+k\alpha-\beta}\rho
    -\lambda f
      \left(\frac{\widetilde\rho}{\lambda}\right)
      \sigma_k^\alpha
      \left(h\left[\frac{\widetilde\rho}{\lambda}\right]\right)v\\
  &=\gamma_0\lambda^{1+k\alpha-\beta}\widetilde\rho
    -\lambda f
      \left(\frac{\widetilde\rho}{\lambda}\right)
       \sigma_k^\alpha
       \left(h\left[\frac{\widetilde\rho}{\lambda}\right]\right)
       \left(
         1+\frac{|\nabla\widetilde\rho|^2}
           {\lambda^2\sinh^2(\widetilde\rho/\lambda)}
      \right)^{1/2}.
  \end{aligned}
\end{equation}
Combining \eqref{eq:radial-flow},
\eqref{eq:time-and-lambda-identities}, and \eqref{eq:t-derivative-before-curvature-scaling},
we obtain
\begin{equation}\label{eq:normalized-radial-flow-exact}
  \begin{aligned}
  \partial_\tau\widetilde\rho
  &=\frac{dt}{d\tau}\partial_t\widetilde\rho
   =\lambda^{\beta-1-k\alpha}\partial_t\widetilde\rho\\
  &=\gamma_0\widetilde\rho
   -\lambda^{\beta-k\alpha} f
       \left(\frac{\widetilde\rho}{\lambda}\right)
     \sigma_k^\alpha
       \left(h\left[\frac{\widetilde\rho}{\lambda}\right]\right)
     \left(
       1+\frac{|\nabla\widetilde\rho|^2}
         {\lambda^2\sinh^2(\widetilde\rho/\lambda)}
    \right)^{1/2}.
  \end{aligned}
\end{equation}

Define
\begin{equation}\label{eq:normalized-weingarten-map-definition}
  \hat h_i{}^j:=\lambda^{-1}h_i{}^j.
\end{equation}
Thus $\hat h_i{}^j$ is the rescaled Weingarten map of
$M_{t(\tau)}$; it is not the Weingarten map of the auxiliary graph
$\widetilde M_\tau$.  By the homogeneity of $\sigma_k$,
\[
  \sigma_k^\alpha
  \left(h\left[\frac{\widetilde\rho}{\lambda}\right]\right)
  =\lambda^{k\alpha}\sigma_k^\alpha(\hat h).
\]
Consequently, the normalized radial function satisfies
\begin{equation}\label{eq:normalized-radial-flow-hat-h}
  \begin{aligned}
  \partial_\tau\widetilde\rho
  ={}&\gamma_0\widetilde\rho
   -\lambda^\beta f
      \left(\frac{\widetilde\rho}{\lambda}\right)
     \sigma_k^\alpha(\hat h)
     \left(
       1+\frac{|\nabla\widetilde\rho|^2}
         {\lambda^2\sinh^2(\widetilde\rho/\lambda)}
     \right)^{1/2}.
  \end{aligned}
\end{equation}
\subsection{The evolution equations}
In this subsection, we record the basic evolution equations for the
normalized Weingarten map, the normalized speed, and the inverse
normalized Weingarten map. Since the auxiliary graph $\widetilde M_\tau$ is not a
geometric rescaling of the original hypersurface, all geometric
quantities below are computed on the unnormalized flow $M_{t(\tau)}$.

All geometric evolution identities in this subsection are computed
in the original normal parameterization of the flow.  In contrast,
the scalar radial equation \eqref{eq:radial-flow} is written at fixed
angular coordinates.  The corresponding radial time derivatives are
$-\Phi/v$ and $-\Phi v$, respectively.  Indeed, in the radial
parameterization,
\[
  \partial_tX=-\Phi\nu+Z,
  \qquad
  Z=-\Phi v\nabla^M\rho
\]
is tangential.  Thus scalar geometric evolution equations acquire a
tangential transport term when pulled back to fixed angular
coordinates; this term vanishes at spatial extrema of the scalar
quantity under consideration.

For later use, let $z$ be an independent scalar argument held fixed
under $\partial_\tau$.  Since $d\lambda/d\tau=\gamma_0\lambda$ by
\eqref{eq:time-and-lambda-identities}, one has
\begin{equation}\label{eq:lambda-f-tau-derivative}
  \partial_\tau\bigl(\lambda^\beta f(z/\lambda)\bigr)
  =\gamma_0\lambda^\beta\bigl(\beta f(r)-rf'(r)\bigr),
  \qquad r=z/\lambda,
\end{equation}
valid in both regimes.

\begin{lemma}[Evolution of the radial distance]
\label{lem:unnormalized-radial-distance-evolution-hyperbolic}
Let $\bar\rho$ be the ambient radial distance from the origin and set
\[
  \rho(p,t):=\bar\rho(X(p,t)).
\]
Along the unnormalized flow~\eqref{eq:unnormalized-flow}, the radial
distance satisfies
\begin{equation}\label{eq:unnormalized-radial-distance-evolution-hyperbolic}
  \partial_t\rho=-\frac{\Phi}{v}.
\end{equation}
\end{lemma}

\begin{proof}
Denote by $\bar\nabla$ the Levi--Civita connection of
$\mathbb H^{n+1}$. Then
$\bar\nabla\bar\rho=\partial_\rho$.

Since $\langle\partial_\rho,\nu\rangle=v^{-1}$, it follows
\[
  \partial_t\rho
  =\left\langle
      \overline\nabla\overline\rho,\partial_tX
    \right\rangle
  =-\Phi\langle\partial_\rho,\nu\rangle
  =-\frac{\Phi}{v}.
\]
\end{proof}

Define the normalized metric and speed by
\begin{equation}\label{eq:preliminary-normalized-metric-and-speed-hyperbolic}
  \widetilde g:=\lambda^2g,
  \qquad
  \widetilde\Phi:=\lambda^{\beta-k\alpha}\Phi.
\end{equation}
The Levi--Civita connections of $g$ and $\widetilde g$ agree at each
fixed time because $\lambda$ is spatially constant. We write this connection as $\widetilde\nabla$.

\begin{lemma}[Evolution of the normalized Weingarten map]
\label{lem:basic-normalized-weingarten-evolution-hyperbolic}
The normalized Weingarten map
$\hat h_i{}^j=\lambda^{-1}h_i{}^j$ satisfies
\begin{equation}\label{eq:basic-normalized-weingarten-evolution-hyperbolic}
  \partial_\tau\hat h_i{}^j
  =\widetilde\nabla^j\widetilde\nabla_i\widetilde\Phi
   +\widetilde\Phi\hat h_i{}^l\hat h_l{}^j
   -\lambda^{-2}\widetilde\Phi\delta_i{}^j
   -\gamma_0\hat h_i{}^j.
\end{equation}
\end{lemma}

\begin{proof}
Along the unnormalized flow $\partial_tX=-\Phi\nu$ in
$\mathbb H^{n+1}$,
\begin{equation}\label{eq:preliminary-unnormalized-h-evolution-hyperbolic}
  \partial_t h_i{}^j
  =\nabla^j\nabla_i\Phi
   +\Phi h_i{}^l h_l{}^j
   -\Phi\delta_i{}^j.
\end{equation}
Consequently,
\[
  \partial_t\hat h_i{}^j
  =-\frac{\lambda'}{\lambda}\hat h_i{}^j
   +\lambda^{-1}\nabla^j\nabla_i\Phi
   +\lambda^{-1}\Phi h_i{}^l h_l{}^j
   -\lambda^{-1}\Phi\delta_i{}^j.
\]
Using
\[
  h_i{}^j=\lambda\hat h_i{}^j,
  \qquad
  \Phi=\lambda^{k\alpha-\beta}\widetilde\Phi,
  \qquad
  \nabla^j\nabla_i\Phi
  =\lambda^{2+k\alpha-\beta}
   \widetilde\nabla^j\widetilde\nabla_i\widetilde\Phi,
\]
together with
\[
  \frac{\lambda'}{\lambda}
  =\gamma_0\lambda^{1+k\alpha-\beta},
  \qquad
  \frac{dt}{d\tau}=\lambda^{\beta-1-k\alpha},
\]
proves the assertion.
\end{proof}

For the speed evolution, put
\begin{equation}\label{eq:preliminary-speed-notation-hyperbolic}
  F:=\sigma_k^\alpha(\hat h),
  \qquad
  \widetilde\Phi=\lambda^\beta f(\rho)F,
\end{equation}
and denote
\begin{equation}\label{eq:derivatives-of-sigma-k}
  \dot\sigma_k^{pq}
  :=\widetilde g^{pl}
    \frac{\partial\sigma_k}{\partial A_q{}^l}(\hat{h}),\quad
  \ddot\sigma_k^{pq,rs}
  :=\widetilde g^{pl}\widetilde g^{rm}
    \frac{\partial^2\sigma_k}
    {\partial A_q{}^l\partial A_s{}^m}(\hat{h}).
\end{equation}
Accordingly,
\[
  \dot F^{pq}
  =\alpha\sigma_k^{\alpha-1}\dot\sigma_k^{pq}.
\]
Define
\begin{equation}\label{eq:preliminary-linearized-operator-hyperbolic}
  \mathcal L
  :=\partial_\tau
   -\lambda^\beta f(\rho)\dot F^{ij}
    \widetilde\nabla_i\widetilde\nabla_j.
\end{equation}

\begin{lemma}[Evolution of the normalized speed]
\label{lem:basic-normalized-speed-evolution-hyperbolic}
The normalized speed satisfies
\begin{equation}\label{eq:basic-normalized-speed-evolution-hyperbolic}
  \begin{aligned}
  \mathcal L\widetilde\Phi
  ={}&(\beta-k\alpha)\gamma_0\widetilde\Phi
     -\frac{f'(\rho)}{\lambda f(\rho)v}\widetilde\Phi^2+\lambda^\beta f(\rho)\widetilde\Phi\dot F^{ij}
      (\hat h^2)_{ij}
     -\lambda^{\beta-2}f(\rho)\widetilde\Phi
      \dot F^{ij}\widetilde g_{ij}.
  \end{aligned}
\end{equation}
\end{lemma}

\begin{proof}
We first calculate the evolution of $\Phi$ along~\eqref{eq:unnormalized-flow}.

By~\eqref{eq:unnormalized-radial-distance-evolution-hyperbolic} and
\eqref{eq:preliminary-unnormalized-h-evolution-hyperbolic},
\begin{equation}\label{eq:unnormalized-speed-evolution-hyperbolic}
  \begin{aligned}
  \partial_t\Phi
  & =f'(\rho)\rho_t\sigma_k^\alpha(h)
     +f(\rho)
      \frac{\partial\sigma_k^\alpha}{\partial A_i{}^j}(h)
      \partial_t h_i{}^j\\
  &=-\frac{f'(\rho)}{f(\rho)v}\Phi^2
     +f(\rho)
      \frac{\partial\sigma_k^\alpha}{\partial A_i{}^j}(h)
      \left(
        \nabla^j\nabla_i\Phi
        +\Phi h_i{}^l h_l{}^j
        -\Phi\delta_i{}^j
      \right),
  \end{aligned}
\end{equation}
Recall $\widetilde\Phi=\lambda^{\beta-k\alpha}\Phi$.
Using \eqref{eq:time-and-lambda-identities}, we therefore obtain
\begin{equation}\label{eq:normalized-speed-chain-rule-hyperbolic}
  \partial_\tau\widetilde\Phi
  =(\beta-k\alpha)\gamma_0\widetilde\Phi
   +\lambda^{2\beta-1-2k\alpha}\partial_t\Phi.
\end{equation}

Since
\[
  h_i{}^j=\lambda\hat h_i{}^j,
  \qquad
  \Phi=\lambda^{k\alpha-\beta}\widetilde\Phi,
  \qquad
  \nabla^j\nabla_i\Phi
  =\lambda^{2+k\alpha-\beta}
    \widetilde\nabla^j\widetilde\nabla_i\widetilde\Phi,
\]
and
\[
  \frac{\partial\sigma_k^\alpha}{\partial A_i{}^j}(h)
  =\lambda^{k\alpha-1}
   \frac{\partial\sigma_k^\alpha}{\partial A_i{}^j}(\hat h),
\]
equation \eqref{eq:unnormalized-speed-evolution-hyperbolic} becomes
\begin{equation}\label{eq:scaled-unnormalized-speed-evolution-hyperbolic}
  \begin{aligned}
  \lambda^{2\beta-1-2k\alpha}\partial_t\Phi
  ={}&\lambda^\beta f(\rho)\dot F^{ij}
      \widetilde\nabla_i\widetilde\nabla_j\widetilde\Phi
     -\frac{f'(\rho)}{\lambda f(\rho)v}\widetilde\Phi^2\\
   &+\lambda^\beta f(\rho)\widetilde\Phi\dot F^{ij}
      (\hat h^2)_{ij}
     -\lambda^{\beta-2}f(\rho)\widetilde\Phi
      \dot F^{ij}\widetilde g_{ij}.
  \end{aligned}
\end{equation}
Substituting
\eqref{eq:scaled-unnormalized-speed-evolution-hyperbolic} into
\eqref{eq:normalized-speed-chain-rule-hyperbolic} proves
\eqref{eq:basic-normalized-speed-evolution-hyperbolic}.
\end{proof}

\begin{lemma}[Evolution of the normalized support function]
\label{lem:basic-normalized-support-evolution-hyperbolic}
For
\begin{equation}\label{eq:preliminary-normalized-support-hyperbolic}
  \widetilde u:=\lambda u
  =\frac{\lambda\sinh\rho}{v},
\end{equation}
we have
\begin{equation}\label{eq:basic-normalized-support-evolution-hyperbolic}
  \begin{aligned}
  \mathcal L\widetilde u
  ={}&\gamma_0\widetilde u
   -(1+k\alpha)\cosh\rho\,\widetilde\Phi
   +\lambda^\beta f(\rho)\widetilde u\dot F^{ij}
      (\hat h^2)_{ij}\\
  &+\lambda^\beta f'(\rho)F\sinh\rho\,(1-v^{-2}).
  \end{aligned}
\end{equation}
\end{lemma}

\begin{proof}
Recall
\[
  u=\langle V,\nu\rangle,
  \qquad
  V=\sinh\rho\,\partial_\rho,
  \qquad
  \overline\nabla V=\cosh\rho\,\mathrm{Id}.
\]
Define
\begin{equation}\label{eq:tangential-components-of-V-hyperbolic}
  V_p:=\overline g(V,X_p),
  \qquad
  V^p:=g^{pq}V_q.
\end{equation}
Therefore,
\begin{equation}\label{eq:first-derivative-support-hyperbolic}
  \begin{aligned}
  \nabla_i u
  &=\left\langle\overline\nabla_{X_i}V,\nu\right\rangle
    +\left\langle V,\overline\nabla_{X_i}\nu\right\rangle\\
  &=h_i{}^pV_p,
  \end{aligned}
\end{equation}
and
\[
  \nabla_j\nabla_i u
  =V^p\nabla_p h_{ij}
   +\cosh\rho\,h_{ij}
   -u(h^2)_{ij}.
\]
We next calculate the time derivative. Along the unnormalized flow~\eqref{eq:unnormalized-flow}, the unit normal satisfies
$\partial_t\nu=\nabla\Phi$. Therefore,
\begin{equation}\label{eq:unnormalized-support-time-derivative-hyperbolic}
  \begin{aligned}
  \partial_tu
  &=\left\langle
      \overline\nabla_{\partial_tX}V,\nu
    \right\rangle
    +\left\langle V,\partial_t\nu\right\rangle\\
  &=-\Phi\cosh\rho+V^p\nabla_p\Phi.
  \end{aligned}
\end{equation}
By
$$\nabla_p\Phi = f'(\rho)\n_p\rho \sigma_k^\alpha(h) + \alpha f(\rho)\sigma_k^{\alpha-1}\frac{\p\sigma_k}{\p A_i{ }^j}(h)\n_ph_i{ }^j,$$
we have
\begin{equation}\label{eq:unnormalized-support-evolution-expanded-hyperbolic}
  \begin{aligned}
  \partial_tu
  ={}&\alpha f(\rho)\sigma_k^{\alpha-1}
      \frac{\partial\sigma_k}{\partial A_i{}^l}(h)
      g^{lj}\nabla_j\nabla_i u\\
   &-\alpha f(\rho)\sigma_k^{\alpha-1}
      \frac{\partial\sigma_k}{\partial A_i{}^l}(h)
      g^{lj}\cosh\rho\,h_{ij}\\
   &+\alpha f(\rho)\sigma_k^{\alpha-1}
      \frac{\partial\sigma_k}{\partial A_i{}^l}(h)
      g^{lj}u(h^2)_{ij}
      -\Phi\cosh\rho\\
   &+f'(\rho)\sigma_k^\alpha(h)
  \sinh\rho\,(1-v^{-2})
  \end{aligned}
\end{equation}
where we used $V^p\nabla_p\rho=\sinh\rho\,(1-v^{-2})$.

Since
\begin{equation*}\label{eq:support-normalization-relations-hyperbolic}
  \begin{gathered}
  \widetilde\nabla=\nabla,
  \quad
  \widetilde u=\lambda u,
  \quad
  \hat h_{ij}=\lambda h_{ij},
  \quad
  (\hat h^2)_{ij}=(h^2)_{ij},\\
  \Phi=\lambda^{k\alpha-\beta}\widetilde\Phi,\quad
  \alpha\sigma_k^{\alpha-1}(h)
  \frac{\partial\sigma_k}{\partial A_i{}^l}(h)g^{lj}
  =\lambda^{1+k\alpha}\dot F^{ij}.
  \end{gathered}
\end{equation*}
It follows from
\eqref{eq:unnormalized-support-evolution-expanded-hyperbolic} that
\begin{equation}\label{eq:normalized-support-time-derivative-hyperbolic}
  \begin{aligned}
  \partial_\tau\widetilde u
  &=\lambda_\tau u
    +\lambda\frac{dt}{d\tau}\partial_tu\\
  ={}&\gamma_0\widetilde u
    +\lambda^\beta f(\rho)\dot F^{ij}
      \widetilde\nabla_j\widetilde\nabla_i\widetilde u
    -\lambda^\beta f(\rho)\cosh\rho\,
      \dot F^{ij}\hat h_{ij}\\
   &+\lambda^\beta f(\rho)\widetilde u\dot F^{ij}
      (\hat h^2)_{ij}
    -\widetilde\Phi\cosh\rho
    +\lambda^\beta f'(\rho)F\sinh\rho\,(1-v^{-2})\\
  ={}&\lambda^\beta f(\rho)\dot F^{ij}
      \widetilde\nabla_i\widetilde\nabla_j\widetilde u
    +\gamma_0\widetilde u
    -(1+k\alpha)\cosh\rho\,\widetilde\Phi\\
   &+\lambda^\beta f(\rho)\widetilde u\dot F^{ij}
      (\hat h^2)_{ij}
    +\lambda^\beta f'(\rho)F\sinh\rho\,(1-v^{-2}).
  \end{aligned}
\end{equation}
Moving the Hessian term to the left-hand side proves
\eqref{eq:basic-normalized-support-evolution-hyperbolic}.
\end{proof}
Applying Simons' identity to
commute the second covariant derivatives of the second fundamental
form, we obtain refined evolution equations for the normalized
Weingarten map and its inverse.
\begin{lemma}
\begin{equation}\label{eq:full-normalized-weingarten-evolution}
  \begin{aligned}
  \mathcal L\hat h_i{}^j
  ={}&
  \lambda^\beta f(\rho)
  \alpha\sigma_k^{\alpha-1}\ddot\sigma_k^{pq,rs}
  \widetilde\nabla^j\hat h_{pq}
  \widetilde\nabla_i\hat h_{rs}
  \\
  &+
  \lambda^\beta f(\rho)
  \alpha(\alpha-1)\sigma_k^{\alpha-2}
  \dot\sigma_k^{pq}\dot\sigma_k^{rs}
  \widetilde\nabla^j\hat h_{pq}
  \widetilde\nabla_i\hat h_{rs}
  \\
  &+
  \lambda^\beta f'(\rho)
  \alpha\sigma_k^{\alpha-1}\dot\sigma_k^{pq}
  \Bigl(
    \widetilde\nabla^j\rho\,
    \widetilde\nabla_i\hat h_{pq}
    +\widetilde\nabla_i\rho\,
    \widetilde\nabla^j\hat h_{pq}
  \Bigr)
  \\
  &+
  \lambda^\beta f'(\rho)
  \sigma_k^\alpha
  \widetilde\nabla^j\widetilde\nabla_i\rho+
  \lambda^\beta f''(\rho)
  \sigma_k^\alpha
  \widetilde\nabla^j\rho\,
  \widetilde\nabla_i\rho
  \\
  &+
  \lambda^\beta f(\rho)
  \alpha\sigma_k^{\alpha-1}
  \left(
    \dot\sigma_k^{pq}(\hat h^2)_{pq}
    +\lambda^{-2}\dot\sigma_k^{pq}\widetilde g_{pq}
  \right)\hat h_i{}^j
  \\
  &+
  \lambda^\beta f(\rho)
  (1-k\alpha)\sigma_k^\alpha(\hat h^2)_i{}^j
  -
  \lambda^{\beta-2} f(\rho)
  (1+k\alpha)\sigma_k^\alpha\delta_i{}^j
  -\gamma_0\hat h_i{}^j.
  \end{aligned}
\end{equation}
\end{lemma}

\begin{proof}
Recall that
\[
  \widetilde\Phi
  =\lambda^\beta f(\rho)\sigma_k^\alpha(\hat h).
\]
Applying the product and chain rules twice, we obtain
\begin{equation}\label{eq:lem25-expanded-speed-hessian-hyperbolic}
  \begin{aligned}
  \widetilde\nabla^j\widetilde\nabla_i\widetilde\Phi
  ={}&
  \lambda^\beta f(\rho)
  \alpha\sigma_k^{\alpha-1}\dot\sigma_k^{pq}
  \widetilde\nabla^j\widetilde\nabla_i\hat h_{pq}
  +
  \lambda^\beta f(\rho)
  \alpha\sigma_k^{\alpha-1}\ddot\sigma_k^{pq,rs}
  \widetilde\nabla^j\hat h_{pq}
  \widetilde\nabla_i\hat h_{rs}
  \\
  &+
  \lambda^\beta f(\rho)
  \alpha(\alpha-1)\sigma_k^{\alpha-2}
  \dot\sigma_k^{pq}\dot\sigma_k^{rs}
  \widetilde\nabla^j\hat h_{pq}
  \widetilde\nabla_i\hat h_{rs}
  \\
  &+
  \lambda^\beta f'(\rho)
  \alpha\sigma_k^{\alpha-1}\dot\sigma_k^{pq}
  \Bigl(
    \widetilde\nabla^j\rho\,
    \widetilde\nabla_i\hat h_{pq}
    +\widetilde\nabla_i\rho\,
    \widetilde\nabla^j\hat h_{pq}
  \Bigr)
  \\
  &+
  \lambda^\beta f'(\rho)
  \sigma_k^\alpha
  \widetilde\nabla^j\widetilde\nabla_i\rho
  +
  \lambda^\beta f''(\rho)
  \sigma_k^\alpha
  \widetilde\nabla^j\rho\,
  \widetilde\nabla_i\rho.
  \end{aligned}
\end{equation}

For a hypersurface in $\mathbb H^{n+1}$, the standard commutation
formula for the second covariant derivatives of its second fundamental
form is
\begin{equation}\label{eq:lem25-unscaled-commutation-formula-hyperbolic}
  \begin{aligned}
  \nabla^j\nabla_i h_{pq}
  ={}&\nabla_p\nabla_q h_i{}^j
      +h_i{}^j(h^2)_{pq}
      -h_{pq}(h^2)_i{}^j\\
     &+h_q{}^j(h^2)_{pi}
      -h_{pi}(h^2)_q{}^j\\
     &-\delta_i{}^j h_{pq}
      +g_{pq}h_i{}^j
      -\delta_q{}^j h_{pi}
      +g_{pi}h_q{}^j.
  \end{aligned}
\end{equation}
Since $\widetilde g=\lambda^2g$ and
$\hat h_i{}^j=\lambda^{-1}h_i{}^j$, we have, as $(0,2)$-tensors,
$\hat h_{ij}=\lambda h_{ij}$. Therefore, rescaling
\eqref{eq:lem25-unscaled-commutation-formula-hyperbolic} gives
\begin{equation}\label{eq:lem25-scaled-commutation-formula-hyperbolic}
  \begin{aligned}
  \widetilde\nabla^j\widetilde\nabla_i\hat h_{pq}
  ={}&\widetilde\nabla_p\widetilde\nabla_q\hat h_i{}^j
      +\hat h_i{}^j(\hat h^2)_{pq}
      -\hat h_{pq}(\hat h^2)_i{}^j\\
     &+\hat h_q{}^j(\hat h^2)_{pi}
      -\hat h_{pi}(\hat h^2)_q{}^j\\
     &-\lambda^{-2}\Bigl(
        \delta_i{}^j\hat h_{pq}
        -\widetilde g_{pq}\hat h_i{}^j
        +\delta_q{}^j\hat h_{pi}
        -\widetilde g_{pi}\hat h_q{}^j
      \Bigr).
  \end{aligned}
\end{equation}
Euler's identity gives
\begin{equation}\label{eq:lem25-euler-identity-hyperbolic}
  \dot\sigma_k^{pq}\hat h_{pq}=k\sigma_k(\hat{h}).
\end{equation}
Moreover, since $\dot\sigma_k$ commutes with $\hat h$, the mixed
terms in \eqref{eq:lem25-scaled-commutation-formula-hyperbolic} satisfy
\[
  \dot\sigma_k^{pq}
  \bigl(
    \hat h_q{}^j(\hat h^2)_{pi}
    -\hat h_{pi}(\hat h^2)_q{}^j
  \bigr)=0
\]
and
\[
  \dot\sigma_k^{pq}
  \bigl(
    \delta_q{}^j\hat h_{pi}
    -\widetilde g_{pi}\hat h_q{}^j
  \bigr)=0.
\]

Substituting \eqref{eq:lem25-scaled-commutation-formula-hyperbolic} into
the first term on the right-hand side of
\eqref{eq:lem25-expanded-speed-hessian-hyperbolic}, and then using
\eqref{eq:basic-normalized-weingarten-evolution-hyperbolic}, we obtain~\eqref{eq:full-normalized-weingarten-evolution}.
\end{proof}

\begin{lemma}[Evolution of the inverse Weingarten map]
\label{lem:inverse-weingarten-evolution-hyperbolic}
Whenever $\hat h_i{}^j$ is positive definite, let
$\check h_i{}^j$ denote its inverse and set
$(\check h^2)_i{}^j:=\check h_i{}^p\check h_p{}^j$. Then
\begin{equation}\label{eq:full-inverse-weingarten-evolution-hyperbolic}
  \begin{aligned}
  \mathcal L\check h_i{}^j
  ={}&-
  \lambda^\beta f(\rho)
  \alpha\sigma_k^{\alpha-1}\ddot\sigma_k^{pq,rs}
  \check h_i{}^a\check h_b{}^j
  \widetilde\nabla^b\hat h_{pq}
  \widetilde\nabla_a\hat h_{rs}
  \\
  &-
  \lambda^\beta f(\rho)
  \alpha(\alpha-1)\sigma_k^{\alpha-2}
  \dot\sigma_k^{pq}\dot\sigma_k^{rs}
  \check h_i{}^a\check h_b{}^j
  \widetilde\nabla^b\hat h_{pq}
  \widetilde\nabla_a\hat h_{rs}
  \\
  &-2\lambda^\beta f(\rho)
  \alpha\sigma_k^{\alpha-1}\dot\sigma_k^{rs}
  \check h_i{}^a
  (\widetilde\nabla_r\hat h_a{}^b)\check h_b{}^p
  (\widetilde\nabla_s\hat h_p{}^q)\check h_q{}^j
  \\
  &-
  \lambda^\beta f'(\rho)
  \alpha\sigma_k^{\alpha-1}\dot\sigma_k^{pq}
  \check h_i{}^a\check h_b{}^j
  \Bigl(
    \widetilde\nabla^b\rho\,
    \widetilde\nabla_a\hat h_{pq}
    +\widetilde\nabla_a\rho\,
    \widetilde\nabla^b\hat h_{pq}
  \Bigr)
  \\
  &-
  \lambda^\beta f'(\rho)\sigma_k^\alpha
  \check h_i{}^a\check h_b{}^j
  \widetilde\nabla^b\widetilde\nabla_a\rho
  -
  \lambda^\beta f''(\rho)\sigma_k^\alpha
  \check h_i{}^a\check h_b{}^j
  \widetilde\nabla^b\rho\,
  \widetilde\nabla_a\rho
  \\
  &-
  \lambda^\beta f(\rho)
  \alpha\sigma_k^{\alpha-1}
  \left(
    \dot\sigma_k^{pq}(\hat h^2)_{pq}
    +\lambda^{-2}\dot\sigma_k^{pq}\widetilde g_{pq}
  \right)\check h_i{}^j
  \\
  &+
  \lambda^\beta f(\rho)
  (k\alpha-1)\sigma_k^\alpha\delta_i{}^j
  +
  \lambda^{\beta-2}f(\rho)
  (1+k\alpha)\sigma_k^\alpha(\check h^2)_i{}^j
  +\gamma_0\check h_i{}^j.
  \end{aligned}
\end{equation}
\end{lemma}

\begin{proof}
The defining identity for the inverse Weingarten map is
\begin{equation}\label{eq:lem26-inverse-weingarten-identity-hyperbolic}
  \check h_i{}^p\hat h_p{}^j=\delta_i{}^j.
\end{equation}
Differentiating \eqref{eq:lem26-inverse-weingarten-identity-hyperbolic}
with respect to $\tau$ and in the spatial directions gives
\begin{equation}\label{eq:lem26-first-derivatives-inverse-weingarten-hyperbolic}
  \begin{aligned}
  \partial_\tau\check h_i{}^j
  &=-\check h_i{}^p
    (\partial_\tau\hat h_p{}^q)\check h_q{}^j,\\
  \widetilde\nabla_r\check h_i{}^j
  &=-\check h_i{}^p
    (\widetilde\nabla_r\hat h_p{}^q)\check h_q{}^j.
  \end{aligned}
\end{equation}
Differentiating the second identity once more, we obtain
\begin{equation}\label{eq:lem26-second-derivative-inverse-weingarten-hyperbolic}
  \begin{aligned}
  \widetilde\nabla_s\widetilde\nabla_r\check h_i{}^j
  ={}&-\check h_i{}^p
      (\widetilde\nabla_s\widetilde\nabla_r\hat h_p{}^q)
      \check h_q{}^j\\
     &+\check h_i{}^a
      (\widetilde\nabla_s\hat h_a{}^b)\check h_b{}^p
      (\widetilde\nabla_r\hat h_p{}^q)\check h_q{}^j\\
     &+\check h_i{}^p
      (\widetilde\nabla_r\hat h_p{}^q)\check h_q{}^a
      (\widetilde\nabla_s\hat h_a{}^b)\check h_b{}^j.
  \end{aligned}
\end{equation}
Since the coefficient
$\lambda^\beta f(\rho)\alpha\sigma_k^{\alpha-1}
\dot\sigma_k^{rs}$ is symmetric in $r,s$, contracting
\eqref{eq:lem26-second-derivative-inverse-weingarten-hyperbolic} and
using \eqref{eq:lem26-first-derivatives-inverse-weingarten-hyperbolic}
gives
\begin{equation}\label{eq:lem26-inverse-weingarten-operator-identity-hyperbolic}
  \begin{aligned}
  \mathcal L\check h_i{}^j
  ={}&-\check h_i{}^p
      (\mathcal L\hat h_p{}^q)\check h_q{}^j\\
     &-2\lambda^\beta f(\rho)
       \alpha\sigma_k^{\alpha-1}\dot\sigma_k^{rs}
       \check h_i{}^a
       (\widetilde\nabla_r\hat h_a{}^b)\check h_b{}^p
       (\widetilde\nabla_s\hat h_p{}^q)\check h_q{}^j.
  \end{aligned}
\end{equation}
Finally, substitute \eqref{eq:full-normalized-weingarten-evolution}
into \eqref{eq:lem26-inverse-weingarten-operator-identity-hyperbolic}
and use
\[
  \check h_i{}^a\hat h_a{}^b\check h_b{}^j
  =\check h_i{}^j,
  \qquad
  \check h_i{}^a(\hat h^2)_a{}^b\check h_b{}^j
  =\delta_i{}^j,
  \qquad
  \check h_i{}^a\delta_a{}^b\check h_b{}^j
  =(\check h^2)_i{}^j.
\]
Collecting the resulting terms proves
\eqref{eq:full-inverse-weingarten-evolution-hyperbolic}.
\end{proof}
\section{The \texorpdfstring{$C^0$}{C0} estimate}
In this section, we derive the $C^0$ estimate for the normalized
flow~\eqref{eq:normalized-radial-flow-hat-h}. We first derive a useful growth estimate on $f$.
\begin{lemma}
\label{lem:model-profile-comparison-hyperbolic}
Assume that either
\eqref{eq:C0-supercritical-decomposition-hyperbolic} and
\eqref{eq:C0-supercritical-flatness-hyperbolic} hold, or
\eqref{eq:C0-critical-decomposition-hyperbolic} and
\eqref{eq:C0-critical-remainder-hyperbolic} hold.  Then, for every
$R_0\in(0,\infty)$, there exist constants $0<c\leq C<\infty$,
depending only on $f|_{(0,R_0]}$, such that
\begin{equation}\label{eq:model-profile-two-sided-hyperbolic}
  cr^\beta\leq f(r)\leq Cr^\beta,
  \qquad 0\leq r\leq R_0.
\end{equation}
\end{lemma}

\begin{proof}
Suppose first that
\eqref{eq:C0-supercritical-decomposition-hyperbolic} and
\eqref{eq:C0-supercritical-flatness-hyperbolic} hold.  Then
$m=\lfloor\beta\rfloor$ and $g(0)=0$.
Since $g\in C^{m+1}$ and $g^{(j)}(0)=0$ for $j=0,\ldots,m$,
Taylor's theorem gives
\[
  g(r)=o(r^{\beta}).
\]
If instead \eqref{eq:C0-critical-decomposition-hyperbolic} and
\eqref{eq:C0-critical-remainder-hyperbolic} hold, then
\eqref{eq:C0-critical-remainder-hyperbolic} directly gives
\[
  \frac{g(r)}{r^\beta}=O(r^\delta)\longrightarrow0.
\]
Thus, in either case,
\[
  \frac{f(r)}{r^\beta}
  =\left(\frac{\sinh r}{r}\right)^\beta
   +\frac{g(r)}{r^\beta}
  \longrightarrow1
  \qquad\text{as }r\downarrow0.
\]
Therefore, the ratio $\frac{f(r)}{r^\beta}$ is a positive continuous
function on $[0,R_0]$. Hence it attains a positive minimum and a
finite maximum on this interval.
\end{proof}

\begin{lemma}
\label{lem:critical-profile-hyperbolic}
Assume that \eqref{eq:C0-critical-decomposition-hyperbolic} and
\eqref{eq:C0-critical-remainder-hyperbolic} hold.  For
$r\in(0,\infty)$, define
\begin{equation}\label{eq:critical-profile-factor-hyperbolic}
  A(r):=\frac{f(r)}{r}\coth^{k\alpha}r
\end{equation}
and
\begin{equation}
  E(r) := A(r) - 1.
\end{equation}
Then, for every $R_0\in(0,\infty)$, set
$\delta_0:=\min\{\delta,2\}>0$.  There are constants
$0<a\leq b<\infty$ and $C>0$, depending only on $f|_{(0,R_0]}$,
such that
\begin{equation}\label{eq:critical-profile-error-hyperbolic}
  a\leq A(r)\leq b,
  \qquad
  |E(r)|\leq Cr^{\delta_0},
  \qquad 0<r\leq R_0.
\end{equation}
Consequently, if $z:[0,T)\to(0,R_0]$ satisfies
\[
  0<z(t)\leq C_1e^{-c_1t},
  \qquad 0\leq t<T,
\]
for some $C_1,c_1>0$, then
\begin{equation}\label{eq:critical-profile-composition-integrable-hyperbolic}
  \int_0^T |E(z(t))|\,dt
  \leq\frac{CC_1^{\delta_0}}{c_1\delta_0}.
\end{equation}
\end{lemma}

\begin{proof}
The bounds for $A$ follow directly from
Lemma~\ref{lem:model-profile-comparison-hyperbolic} and the positive
lower and upper bounds for $r\coth r$ on $(0,R_0]$. By~\eqref{eq:C0-critical-decomposition-hyperbolic}, we have
\[
  E(r)
  =\frac{\sinh r}{r}\cosh^{k\alpha}r-1
   +\frac{g(r)}{r}\coth^{k\alpha}r.
\]
Hence,
\begin{equation*}
  \begin{split}
     \frac{|E(r)|}{r^{\delta_0}}\leq{} & \frac{|(\sinh r/r)\cosh^{k\alpha}r-1|}{r^{\delta_0}} + \frac{|g(r)|}{r^{1+\delta_0}}\coth^{k\alpha}r \\
       \leq &  Cr^{2-\delta_0} + Cr^{k\alpha+\delta-\delta_0}\coth^{k\alpha}r
  \end{split}
\end{equation*}
where we used~\eqref{eq:C0-critical-remainder-hyperbolic} and
$(\sinh r/r)\cosh^{k\alpha}r=1+O(r^2)$. Therefore, by our choice of
$\delta_0$, the ratio $\frac{|E(r)|}{r^{\delta_0}}$ is bounded on the
interval $(0,R_0]$.

Finally,
\[
  |E(z(t))|
  \leq CC_1^{\delta_0}e^{-c_1\delta_0t},
\]
which proves \eqref{eq:critical-profile-composition-integrable-hyperbolic}.
\end{proof}

We formulate the $C^0$ estimate as follows.  It is a conditional
estimate in the sense of Section~\ref{subsec:normalized-flow}: it is
derived on any time interval on which the solution remains a radial
graph over $\mathbb S^n$ centred at $o$---equivalently, on which the
evolving hypersurface remains strictly $h$-convex and encloses
$o$---and it does not by itself assert that this representation
persists.  The bootstrap closing the maximal such interval is carried
out at the beginning of Section~\ref{sec:C1-estimate}  and in
Lemma~\ref{lem:h-convexity-and-lower-curvature-hyperbolic}.

\begin{lemma}[$C^0$ estimate]
\label{lem:normalized-C0-estimate-hyperbolic}
Under the assumptions of either Theorem~\ref{thm:main} or
Theorem~\ref{thm:main-critical}, there exist positive constants $c_0$ and
$C_0$, depending only on the initial radial bounds and $f$, such that
\begin{equation}\label{eq:normalized-C0-estimate-hyperbolic}
  0<c_0\leq\widetilde\rho(\theta,\tau)\leq C_0.
\end{equation}
\end{lemma}

\begin{proof}
The scalar equation \eqref{eq:radial-flow} gives
$\partial_t\rho<0$ and hence $0<\rho(\theta,t)\leq R_0$ with $R_0$ the initial radial bound.

\begin{itemize}
  \item Suppose first that
  \eqref{eq:C0-supercritical-decomposition-hyperbolic} and
  \eqref{eq:C0-supercritical-flatness-hyperbolic} hold.  Define
  \begin{equation}\label{eq:normalized-radial-extrema-hyperbolic}
    \widetilde\rho_-(\tau)
    :=\min_{\mathbb S^n}\widetilde\rho(\cdot,\tau),
    \qquad
    \widetilde\rho_+(\tau)
    :=\max_{\mathbb S^n}\widetilde\rho(\cdot,\tau).
  \end{equation}
  At a spatial maximum of $\widetilde\rho$, one has $v=1$ and
  $\nabla^2\widetilde\rho\leq0$.  Since $\lambda$ is spatially
  constant, \eqref{eq:weingarten-map-rho} and
  $\hat h_i{}^j=\lambda^{-1}h_i{}^j$ give
  \begin{equation}\label{eq:normalized-maximum-curvature-comparison-hyperbolic}
    \hat h_i{}^j
    \geq\lambda^{-1}\coth
      \frac{\widetilde\rho_+}{\lambda}\,\delta_i{}^j.
  \end{equation}
  Similarly, at a spatial minimum,
  \begin{equation}\label{eq:normalized-minimum-curvature-comparison-hyperbolic}
    \hat h_i{}^j
    \leq\lambda^{-1}\coth
      \frac{\widetilde\rho_-}{\lambda}\,\delta_i{}^j.
  \end{equation}
  Applying the maximum principle directly to
  \eqref{eq:normalized-radial-flow-hat-h} (with the time derivatives
  of the extrema understood in the sense of the convention stated in
  the Notation section) and setting
  $\mu:=\beta-k\alpha-1>0$, we obtain
  \begin{equation}\label{eq:normalized-extrema-inequalities-hyperbolic}
    \begin{aligned}
    \frac{d}{d\tau}\widetilde\rho_+
    &\leq\gamma_0\widetilde\rho_+
     -\gamma_0\lambda^{1+\mu}
      f\left(\frac{\widetilde\rho_+}{\lambda}\right)
      \coth^{k\alpha}\left(\frac{\widetilde\rho_+}{\lambda}\right),\\
    \frac{d}{d\tau}\widetilde\rho_-
    &\geq\gamma_0\widetilde\rho_-
     -\gamma_0\lambda^{1+\mu}
      f\left(\frac{\widetilde\rho_-}{\lambda}\right)
      \coth^{k\alpha}\left(\frac{\widetilde\rho_-}{\lambda}\right).
    \end{aligned}
  \end{equation}
 By
  \eqref{eq:model-profile-two-sided-hyperbolic} and the
  bounds for $r\coth r$ on $(0,R_0]$, there exist constants
  $0<a_0<a_1$, depending only on $f$ and $R_0$, such that
\begin{equation}\label{eq:normalized-logistic-coefficient-hyperbolic}
  a_0\leq
  \frac{f(r)}{r^\beta}(r\coth r)^{k\alpha}
  \leq a_1,
  \qquad 0<r\leq R_0.
\end{equation}
Consequently, \eqref{eq:normalized-extrema-inequalities-hyperbolic} becomes
\begin{equation}\label{eq:normalized-logistic-inequalities-hyperbolic}
  \begin{aligned}
  \frac{d}{d\tau}\widetilde\rho_+
  &\leq\gamma_0\widetilde\rho_+
       (1-a_0\widetilde\rho_+^\mu),\\
  \frac{d}{d\tau}\widetilde\rho_-
  &\geq\gamma_0\widetilde\rho_-
       (1-a_1\widetilde\rho_-^\mu).
  \end{aligned}
\end{equation}
Therefore the standard comparison principle gives
\begin{equation}\label{eq:normalized-C0-supercritical-hyperbolic}
  \min\{\widetilde\rho_-(0),a_1^{-1/\mu}\}
  \leq\widetilde\rho(\theta,\tau)
  \leq\max\{\widetilde\rho_+(0),a_0^{-1/\mu}\}.
\end{equation}
  \item Suppose next that
  \eqref{eq:C0-critical-decomposition-hyperbolic} and
  \eqref{eq:C0-critical-remainder-hyperbolic} hold.  In this case
  $\mu=0$ and $\tau=t$.  Define the extrema of the unnormalized radial
  function by
  \[
    \rho_-(t):=\min_{\mathbb S^n}\rho(\cdot,t),
    \qquad
    \rho_+(t):=\max_{\mathbb S^n}\rho(\cdot,t),
  \]
  Then $\widetilde\rho_\pm=\lambda\rho_\pm$.  Let $A$ and $E$ be
  the functions defined in Lemma~\ref{lem:critical-profile-hyperbolic}.
  
  At a spatial maximum of $\rho$, one has $\nabla\rho=0$ and
  $\nabla^2\rho\leq0$, and hence
  \[
    h_i{}^j[\rho]\geq\coth\rho_+\,\delta_i{}^j.
  \]
  Applying \eqref{eq:radial-flow} at every differentiability time of
  $\rho_+$ gives
  \begin{equation}\label{eq:critical-unnormalized-maximum-hyperbolic}
    \frac{d}{dt}\rho_+
    \leq-\gamma_0 f(\rho_+)\coth^{k\alpha}\rho_+
    =-\gamma_0A(\rho_+)\rho_+.
  \end{equation}
  By Lemma~\ref{lem:critical-profile-hyperbolic}, $A\geq a>0$ on
  $(0,R_0]$.  Thus the ODE comparison yields
  \begin{equation}\label{eq:critical-unnormalized-extrema-decay-hyperbolic}
    0<\rho_+\leq\rho_+(0)e^{-a\gamma_0t}.
  \end{equation}
  The last assertion of Lemma~\ref{lem:critical-profile-hyperbolic} now
  gives
  \begin{equation}\label{eq:critical-upper-error-integrable-hyperbolic}
    \int_0^T |E(\rho_+(t))|\,dt\leq C,
  \end{equation}
  where $C$ is independent of $T$.
  Since $\lambda_t=\gamma_0\lambda$, equation
  \eqref{eq:critical-unnormalized-maximum-hyperbolic} implies
  \begin{equation}\label{eq:critical-lambda-maximum-comparison-hyperbolic}
    \frac{d}{dt}\log(\lambda\rho_+)
    \leq\gamma_0\bigl(1-A(\rho_+)\bigr)
    =-\gamma_0E(\rho_+)
    \leq\gamma_0|E(\rho_+)|.
  \end{equation}
  Integrating and using
  \eqref{eq:critical-upper-error-integrable-hyperbolic}, we obtain
  \[
    \widetilde\rho_+(t)=\lambda(t)\rho_+(t)
    \leq\rho_+(0)
      \exp\left(
        \gamma_0\int_0^T|E(\rho_+(s))|\,ds
      \right)
    \leq C_0.
  \]

  We next estimate the minimum radius.  At a spatial minimum of
  $\rho$, one has $\nabla\rho=0$ and $\nabla^2\rho\geq0$, and hence
  \[
    h_i{}^j[\rho]\leq\coth\rho_-\,\delta_i{}^j.
  \]
  Therefore \eqref{eq:radial-flow} gives, at every differentiability
  time of $\rho_-$,
  \begin{equation}\label{eq:critical-unnormalized-minimum-hyperbolic}
    \frac{d}{dt}\rho_-
    \geq-\gamma_0 f(\rho_-)\coth^{k\alpha}\rho_-
    =-\gamma_0A(\rho_-)\rho_-.
  \end{equation}
  Since $0<\rho_-\leq\rho_+$,
  \eqref{eq:critical-unnormalized-extrema-decay-hyperbolic} and the last
  assertion of Lemma~\ref{lem:critical-profile-hyperbolic} imply
  \begin{equation}\label{eq:critical-lower-error-integrable-hyperbolic}
    \int_0^T |E(\rho_-(t))|\,dt\leq C,
  \end{equation}
  where $C$ is independent of $T$.
  Combining \eqref{eq:critical-unnormalized-minimum-hyperbolic} with
  $\lambda_t=\gamma_0\lambda$, we obtain
  \begin{equation}\label{eq:critical-lambda-minimum-comparison-hyperbolic}
    \frac{d}{dt}\log(\lambda\rho_-)
    \geq\gamma_0\bigl(1-A(\rho_-)\bigr)
    =-\gamma_0E(\rho_-)
    \geq-\gamma_0|E(\rho_-)|.
  \end{equation}
  Integrating and using
  \eqref{eq:critical-lower-error-integrable-hyperbolic}, we obtain
  \[
    \widetilde\rho_-(t)=\lambda(t)\rho_-(t)
    \geq\rho_-(0)
      \exp\left(
        -\gamma_0\int_0^T|E(\rho_-(s))|\,ds
      \right)
    \geq c_0>0.
  \]
  Consequently,
  \[
    c_0\leq\widetilde\rho_-(t)
    \leq\widetilde\rho(\theta,t)
    \leq\widetilde\rho_+(t)\leq C_0,
  \]
  which completes the proof of
  \eqref{eq:normalized-C0-estimate-hyperbolic} in the critical case.
\end{itemize}
\end{proof}

\begin{corollary}\label{cor:lambda-sin-rho-two-sided-hyperbolic}
Under the assumptions of either Theorem~\ref{thm:main} or
Theorem~\ref{thm:main-critical},
there exist constants $c_1,C_1>0$ such that, on any time interval on
which the normalized $C^0$ estimate
\eqref{eq:normalized-C0-estimate-hyperbolic} holds,
\[
  c_1\leq\lambda\sinh\rho\leq C_1.
\]
\end{corollary}

\begin{proof}
Since $0<\rho\leq R_0$, it follows that
\[
  \rho\leq\sinh\rho\leq\frac{\sinh R_0}{R_0}\rho.
\]
Multiplying by $\lambda$ and using
\eqref{eq:normalized-C0-estimate-hyperbolic}, we obtain
\[
  c_0
  \leq\lambda\sinh\rho
  \leq\frac{\sinh R_0}{R_0}C_0,
\]
which proves the assertion.
\end{proof}

\begin{remark}
For the $C^0$ estimate alone, the assumptions on $f$ in the
supercritical case can be weakened. It is enough to assume
\[
  0<\liminf_{r\downarrow0}\frac{f(r)}{r^\beta}
  \leq
  \limsup_{r\downarrow0}\frac{f(r)}{r^\beta}
  <\infty.
\]
Indeed, since $f$ is positive and continuous away from the origin,
these asymptotic bounds give the two-sided estimate
\eqref{eq:model-profile-two-sided-hyperbolic} on every fixed interval
$(0,R_0]$.

In the critical case $\beta=1+k\alpha$, the present argument requires
the stronger decomposition
\[
  f(r)=\sinh^{1+k\alpha}r+g(r),
  \qquad
  g(r)=O\bigl(r^{1+k\alpha+\delta}\bigr)
\]
for some $\delta>0$.  Thus the remainder $g$ must decay strictly faster
than the leading order $r^{1+k\alpha}$ as $r\downarrow0$.  This positive
power gain makes the error $E(\rho_\pm(t))$ integrable in time.  The
weaker condition $g=o(r^{1+k\alpha})$ alone does not, in general,
guarantee this integrability.
\end{remark}

\section{The \texorpdfstring{$C^1$}{C1} estimate}
\label{sec:C1-estimate}

The strict $h$-convexity assumption is imposed only on the initial
hypersurface.  Recall that $T_{\max}\in(0,\infty]$ denotes the maximal
existence time of the smooth solution of \eqref{eq:unnormalized-flow}
(see Section~\ref{subsec:normalized-flow}), and write
\[
  \tau_{\max}:=\tau(T_{\max})
\]
for the corresponding normalized time, with $\tau$ as in
\eqref{eq:normalized-time}.  By continuity, there exists a maximal
normalized time interval $[0,\tau_h)$, $\tau_h\in(0,\tau_{\max}]$,
on which the evolving hypersurface remains strictly $h$-convex
{and} encloses the fixed point $o$:
\begin{equation}\label{eq:enclosure-condition}
  o\in\operatorname{int}\Omega_\tau,
\end{equation}
where $\Omega_\tau\Subset\mathbb H^{n+1}$ denotes the domain enclosed
by $M_{t(\tau)}$.  Both conditions are needed at this stage: the
radial-graph representation and the preceding $C^0$ estimates require
\eqref{eq:enclosure-condition}, while strict $h$-convexity alone does
not establish enclosure of $o$.
Lemma~\ref{lem:h-convexity-and-lower-curvature-hyperbolic} below
shows that $\tau_h=\tau_{\max}$, and
Proposition~\ref{prop:C2-to-Cinfty-normalized-hyperbolic} then shows
that $\tau_{\max}=\infty$.  Recall the hyperbolic support function
$u=\sinh\rho/v$ from \eqref{eq:radial-support-function} and its
normalized counterpart $\widetilde u=\lambda u$ from
\eqref{eq:preliminary-normalized-support-hyperbolic}.

The estimates up to and including
Lemma~\ref{lem:normalized-speed-lower-bound-hyperbolic} are, at this
stage, conditional estimates on $[0,\tau_h)$.  Thus the use of strict
$h$-convexity in those estimates comes from the definition of
$\tau_h$, and is not an assertion that preservation has already been
proved.

Enclosure cannot be the property that terminates the interval.
Indeed, on $[0,\tau_h)$ the normalized $C^0$ estimate of
Lemma~\ref{lem:normalized-C0-estimate-hyperbolic} applies and gives
\begin{equation}\label{eq:enclosure-lower-bound}
  \min_{\mathbb S^n}\rho\bigl(\cdot,t(\tau)\bigr)
  \geq\frac{c_0}{\lambda(\tau)}
  =c_0e^{-\gamma_0\tau},
\end{equation}
where the equality uses $\lambda=e^{\gamma_0\tau}$ from
\eqref{eq:time-and-lambda-identities}.  At a finite endpoint
$\tau_h<\tau_{\max}$ the solution is smooth and this lower bound
stays positive, so the boundary cannot reach $o$; enclosure therefore
persists past $\tau_h$ and cannot terminate the maximal interval.
The first-contact argument in
Lemma~\ref{lem:h-convexity-and-lower-curvature-hyperbolic} then
excludes the loss of strict $h$-convexity at $\tau_h$, which closes
the bootstrap.  Its first-contact proof uses neither the lower speed
estimate nor any of the subsequent curvature estimates.

\begin{lemma}[$C^1$ estimate on the maximal interval $[0,\tau_h)$]
\label{lem:normalized-C1-estimate-hyperbolic}
Under the assumptions of either Theorem~\ref{thm:main} or
Theorem~\ref{thm:main-critical}, there exist constants $c,C>0$,
depending only on the $C^0$ bound of
Lemma~\ref{lem:normalized-C0-estimate-hyperbolic} and independent of
$\tau_h$, such that on $\mathbb S^n\times [0,\tau_h)$,
\begin{equation}\label{eq:normalized-C1-estimate-hyperbolic}
  v\leq C,
  \qquad
  \widetilde u\geq c,
  \qquad
  |\nabla\widetilde\rho|\leq C.
\end{equation}
\end{lemma}

\begin{proof}
Fix $\tau\in [0,\tau_h)$ and let $\theta_0$ be a minimum point of
$\widetilde u(\cdot,\tau)$.  Put
$V=\sinh\rho\,\partial_\rho$.  Since
$\overline\nabla V=\cosh\rho\,\mathrm{Id}$, differentiation of
$u=\langle V,\nu\rangle$ gives
\begin{equation}\label{eq:first-derivative-support-C1-hyperbolic}
  \nabla_i u
  =h_i{}^j\langle V,X_j\rangle.
\end{equation}
Because $\lambda$ is spatially constant, at $\theta_0$ we have
\begin{equation}\label{eq:support-minimum-C1-hyperbolic}
  0=\nabla_i\widetilde u
   =\lambda h_i{}^j\langle V,X_j\rangle.
\end{equation}
On $[0,\tau_h)$, the Weingarten map $h_i{}^j$ is positive definite and
hence invertible.  It follows that
\[
  \langle V,X_j\rangle=0
  \qquad\text{for every }j.
\]
For the radial graph,
\[
  \langle V,X_j\rangle=\sinh\rho\,\rho_j.
\]
Since $\rho>0$, it follows that
$\rho_j(\theta_0,\tau)=0$ for every $j$. Consequently,
$v(\theta_0,\tau)=1$ and
\[
  \min_{\mathbb S^n}\widetilde u(\cdot,\tau)
  =\lambda\sinh\rho(\theta_0,\tau).
\]
Corollary~\ref{cor:lambda-sin-rho-two-sided-hyperbolic} now yields
\begin{equation}\label{eq:normalized-support-lower-bound-C1-hyperbolic}
  \widetilde u\geq c
  \qquad\text{on }\mathbb S^n\times[0,\tau_h).
\end{equation}
The same corollary gives $\lambda\sinh\rho\leq C$, and hence
\[
  v=\frac{\lambda\sinh\rho}{\widetilde u}\leq C.
\]
Finally, the definition of $v$ implies
\[
  |\nabla\widetilde\rho|
  =\lambda|\nabla\rho|
  =\lambda\sinh\rho\sqrt{v^2-1}\leq C.
\]
All constants depend only on the constants in the normalized $C^0$
estimate and are independent of $\tau_h$.
\end{proof}

\section{The \texorpdfstring{$C^2$}{C2} estimate}

The aim of this section is to obtain uniform two-sided bounds for the
normalized principal curvatures and, consequently, a uniform bound for
the spherical Hessian of $\widetilde\rho$.  We proceed in several steps.
Lemmas~\ref{lem:profile-logarithmic-derivative-bound-hyperbolic} and
\ref{lem:profile-second-derivative-bound-hyperbolic} establish the
properties of the radial weight $f$ needed in the $C^2$ estimate.
Lemma~\ref{lem:normalized-speed-lower-bound-hyperbolic} then uses the
normalized $C^0$ and $C^1$ estimates and the maximum principle to obtain
a positive lower bound for the normalized speed.
The main estimate, Lemma~\ref{lem:h-convexity-and-lower-curvature-hyperbolic},
uses inverse concavity, the supporting-horosphere inequality, and the
maximum principle to preserve strict $h$-convexity and obtain a uniform
positive lower bound for every normalized principal curvature.
Lemma~\ref{lem:normalized-speed-upper-bound-hyperbolic} then applies the
maximum principle to the quotient of the normalized speed by a shifted
support function and obtains an upper bound for the normalized speed.
Finally, Lemma~\ref{lem:normalized-C2-estimate-hyperbolic} combines these
two bounds to control all normalized principal curvatures from both sides,
and then uses the radial graph formula to bound
$\nabla_\sigma^2\widetilde\rho$.  At the end of the section, the
classical regularity theory for uniformly parabolic equations with
convex level sets gives long-time existence and all higher-order
estimates.

Throughout this section, all constants are independent of
$\varepsilon_h(M_0)^{-1}$ and of $\tau_{\max}$, except where
explicitly stated.  The only place where the initial strictness margin
$\varepsilon_h(M_0)$ enters is the first-contact argument of
Lemma~\ref{lem:h-convexity-and-lower-curvature-hyperbolic}, and even
there the final bounds depend only on the fixed data listed in the
Notation paragraph of the Introduction.

\begin{lemma}\label{lem:profile-logarithmic-derivative-bound-hyperbolic}
Assume the radial-weight hypotheses of either
Theorem~\ref{thm:main} or Theorem~\ref{thm:main-critical}.  Then
\[
  f'(r)>0,
  \qquad r>0.
\]
Moreover, for every $R_0\in(0,\infty)$, there exist constants
$c,C>0$, depending only on $f|_{(0,R_0]}$, such that
\begin{equation}\label{eq:profile-logarithmic-derivative-bound-hyperbolic}
  cr^{\beta-1}\leq f'(r)\leq Cr^{\beta-1},
  \qquad
  c\leq\frac{r f'(r)}{f(r)}\leq C,
  \qquad 0<r\leq R_0.
\end{equation}
\end{lemma}

\begin{proof}
Set
\[
  p:=1+k\alpha,
  \qquad
  q:=f^{1/p},
  \qquad
  a(r):=\frac12\tanh\frac r2.
\]
The structural condition gives
\begin{equation}\label{eq:radial-weight-q-differential-inequality}
  q''(r)-a(r)q'(r)\geq0.
\end{equation}
Let
\[
  A(r):=\int_0^r a(s)\,ds.
\]
Then
\begin{equation}\label{eq:radial-weight-integrating-factor}
  \bigl(e^{-A(r)}q'(r)\bigr)'
  =e^{-A(r)}\bigl(q''(r)-a(r)q'(r)\bigr)\geq0.
\end{equation}
We claim that $q'(r)>0$ for every $r>0$.  Otherwise, if
$e^{-A(r_0)}q'(r_0)\leq0$ for some $r_0>0$, the monotonicity in
\eqref{eq:radial-weight-integrating-factor} would imply
$q'(r)\leq0$ for $0<r\leq r_0$.  Since $q$ extends continuously to
$q(0)=0$, this would give $q(r_0)\leq0$, contrary to $q(r_0)>0$.
Thus $q'>0$ on $(0,\infty)$, and
\eqref{eq:radial-weight-q-differential-inequality} also gives
$q''\geq a q'>0$.  In particular, $f'=p q^{p-1}q'>0$.

It remains to prove the quantitative estimates.  In the supercritical
case, let $m=\lfloor\beta\rfloor$.  Taylor's theorem and the vanishing
of $g^{(j)}(0)$ for $j=0,\ldots,m$ give
\[
  g'(r)=O(r^m)=o(r^{\beta-1}).
\]
In the critical case, put $N=\lceil\beta+\delta\rceil$.  The condition
$g(r)=O(r^{\beta+\delta})$ together with $g\in C^N([0,\infty))$
implies
\[
  g^{(j)}(0)=0,
  \qquad j=0,\ldots,N-1,
\]
and hence
\[
  g'(r)=O(r^{N-1})=o(r^{\beta-1}).
\]
Since
\[
  (\sinh^\beta r)'
  =\beta\sinh^{\beta-1}r\cosh r
  =\beta r^{\beta-1}+O(r^{\beta+1}),
\]
we obtain, in both cases,
\begin{equation}\label{eq:radial-weight-first-derivative-asymptotics}
  f'(r)=\beta r^{\beta-1}+o(r^{\beta-1}),
  \qquad
  \frac{rf'(r)}{f(r)}\longrightarrow\beta
  \quad\text{as }r\downarrow0.
\end{equation}
The asserted two-sided estimates therefore hold near the origin.  On
every compact subinterval of $(0,R_0]$, they follow from the smoothness
of $f$, the positivity of $f$ and $f'$, and compactness.  Combining the
two regions proves
\eqref{eq:profile-logarithmic-derivative-bound-hyperbolic}.
\end{proof}

\begin{lemma}\label{lem:profile-second-derivative-bound-hyperbolic}
Assume that the conditions of either Theorem~\ref{thm:main} or
Theorem~\ref{thm:main-critical} hold.  Then, for every
$R_0\in(0,\infty)$, there exists a constant $C>0$, depending only on
$f|_{(0,R_0]}$, such that
\begin{equation}\label{eq:profile-second-derivative-bound-hyperbolic}
  |f''(r)|\leq Cr^{\beta-2},
  \qquad 0<r\leq R_0.
\end{equation}
\end{lemma}

\begin{proof}
Write $f(r)=\sinh^\beta r+g(r)$.  In the supercritical case, let
$m=\lfloor\beta\rfloor$.  The vanishing of
$g^{(j)}(0)$ for $j=0,\ldots,m$, together with
$g\in C^{m+1}$, allows us to apply Taylor's theorem to $g''$ and gives
\[
  g''(r)=O(r^{m-1})=o(r^{\beta-2}).
\]

In the critical case, the assumptions of
Theorem~\ref{thm:main-critical}, with
$N=\lceil\beta+\delta\rceil$, give
\[
  g^{(j)}(0)=0,
  \qquad j=0,\ldots,N-1.
\]
Taylor's theorem applied to $g''$ therefore yields
\[
  g''(r)=O(r^{N-2})=o(r^{\beta-2}),
\]
where we used $N>\beta$.

Moreover,
\[
  (\sinh^\beta r)''
  =\beta(\beta-1)\sinh^{\beta-2}r\cosh^2r
   +\beta\sinh^\beta r
  =\beta(\beta-1)r^{\beta-2}+O(r^\beta).
\]
Thus, in both cases,
\[
  f''(r)=\beta(\beta-1)r^{\beta-2}+o(r^{\beta-2})
  \qquad\text{as }r\downarrow0.
\]
The estimate near the origin, together with the smoothness of $f$ on
every compact subinterval of $(0,\infty)$, proves
\eqref{eq:profile-second-derivative-bound-hyperbolic}.
\end{proof}

\begin{lemma}[Lower bound for the normalized speed]
\label{lem:normalized-speed-lower-bound-hyperbolic}
On the maximal interval $[0,\tau_h)$ of Section~\ref{sec:C1-estimate}, there exists
a constant $c_{\widetilde\Phi}>0$, independent of $\tau_h$, such
that
\begin{equation}\label{eq:normalized-speed-lower-bound-hyperbolic}
  \widetilde\Phi\geq c_{\widetilde\Phi}.
\end{equation}
\end{lemma}

\begin{proof}
By Lemma~\ref{lem:model-profile-comparison-hyperbolic}, the normalized
$C^0$ estimate, and $\rho=\widetilde\rho/\lambda$, we have
\[
  c\widetilde\rho^\beta
  \leq\lambda^\beta f(\rho)
  \leq C\widetilde\rho^\beta.
\]
Moreover, since $\widetilde\rho=\lambda\rho$, we have the identity
\[
  \frac{f'(\rho)}{\lambda f(\rho)}
  =\frac{\rho f'(\rho)}{f(\rho)}
    \frac{1}{\lambda\rho}
  =\frac{\rho f'(\rho)}{f(\rho)}
    \frac{1}{\widetilde\rho}.
\]
Hence Lemma~\ref{lem:profile-logarithmic-derivative-bound-hyperbolic},
together with $0<c_0\leq\widetilde\rho\leq C_0$, implies that
\begin{equation}\label{eq:speed-profile-coefficient-bounds}
  0<c\leq\lambda^\beta f(\rho)\leq C,
  \qquad
  0\leq\frac{f'(\rho)}{\lambda f(\rho)}\leq C.
\end{equation}

Let
\[
  \widetilde{\Phi}_{\min}(\tau):=\min_{\mathbb S^n}\widetilde\Phi(\cdot,\tau).
\]
Fix $\tau<\tau_h$. In a $\widetilde g$-orthonormal principal frame,
$\hat\kappa_i=\lambda^{-1}\kappa_i>\lambda^{-1}$ and ellipticity gives
$\partial F/\partial\hat\kappa_i>0$.  Hence, at a spatial minimum of
$\widetilde\Phi$,
\[
  \lambda^\beta f(\rho)\widetilde\Phi
    \dot F^{ij}
    \bigl((\hat h^2)_{ij}-\lambda^{-2}\widetilde g_{ij}\bigr)
  =\lambda^\beta f(\rho)\widetilde\Phi
    \sum_i\frac{\partial F}{\partial\hat\kappa_i}
       \bigl(\hat\kappa_i^2-\lambda^{-2}\bigr)
  >0.
\]
Consequently, the speed evolution equation
\eqref{eq:basic-normalized-speed-evolution-hyperbolic} and
\eqref{eq:speed-profile-coefficient-bounds} give
\begin{equation}\label{eq:minimum-speed-logistic-inequality-hyperbolic}
  \frac{d}{d\tau}\widetilde{\Phi}_{\min}
  \geq(\beta-k\alpha)\gamma_0\widetilde{\Phi}_{\min}-C\widetilde{\Phi}_{\min}^2.
\end{equation}
Since $\beta-k\alpha\geq1$, it follows that
\[
  \widetilde\Phi_{\min}(\tau)
  \geq
  \min\left\{
    \widetilde\Phi_{\min}(0),
    \frac{(\beta-k\alpha)\gamma_0}{C}
  \right\}>0.
\]
This proves \eqref{eq:normalized-speed-lower-bound-hyperbolic}.
\end{proof}
Let
$\hat\kappa_1\leq\cdots\leq\hat\kappa_n$ denote the eigenvalues of
$\hat h$.  By Lemmas~\ref{lem:model-profile-comparison-hyperbolic},
\ref{lem:profile-logarithmic-derivative-bound-hyperbolic}, and
\ref{lem:profile-second-derivative-bound-hyperbolic}, together with the
two-sided $C^0$ estimate, all the normalized coefficients involving
$f$ and its first two derivatives are uniformly bounded.  More
precisely,
\begin{equation}\label{eq:normalized-profile-coefficient-bounds}
  \begin{gathered}
  0<c\leq
  \lambda^\beta f\left(\frac{\widetilde\rho}{\lambda}\right)
  \leq C,\\
  \lambda^{\beta-1}
  \left|f'\left(\frac{\widetilde\rho}{\lambda}\right)\right|
  +\lambda^{\beta-2}
  \left|f''\left(\frac{\widetilde\rho}{\lambda}\right)\right|
  \leq C.
  \end{gathered}
\end{equation}
Since
$\widetilde\Phi=\lambda^\beta
f(\widetilde\rho/\lambda)\sigma_k^\alpha(\hat h)$,
the lower speed bound in
Lemma~\ref{lem:normalized-speed-lower-bound-hyperbolic} and
\eqref{eq:normalized-profile-coefficient-bounds} imply
\begin{equation}\label{eq:sigma-k-lower-bound-before-C2}
  \sigma_k(\hat h)\geq c>0.
\end{equation}

In the proof below, we shall work directly with the largest eigenvalue
$\Lambda:=\lambda_{\max}(\check h)$.  Since $\Lambda$ is in general only
locally Lipschitz, its differential inequality must be interpreted in the
barrier sense. We therefore record the convention that will be used.

\begin{definition*}[Barrier convention]
Let $J$ be a time interval.  A continuous function $\phi$ defined on
$M\times J$ satisfies
\[
  \mathcal L\phi
  \leq\mathcal F(\widetilde\nabla\phi,\phi,x,\tau)
\]
in the barrier sense if, for every $(x_0,\tau_0)\in M\times J$,
there exist a neighbourhood $U$ of $x_0$, a number $\epsilon>0$,
and a smooth function
\[
  \psi\in C^\infty
  \bigl(U\times(\tau_0-\epsilon,\tau_0]\bigr)
\]
such that $\psi\leq\phi$ on $U\times(\tau_0-\epsilon,\tau_0]$, with
equality at $(x_0,\tau_0)$, and
\begin{equation}\label{eq:barrier-subsolution-definition}
  \mathcal L\psi(x_0,\tau_0)
  \leq
  \mathcal F\bigl(
    \widetilde\nabla\psi(x_0,\tau_0),
    \psi(x_0,\tau_0),x_0,\tau_0
  \bigr).
\end{equation}
\end{definition*}

\begin{remark}[Lower supporting
barriers]\label{rem:lower-supporting-barriers}
In the proof of
Lemma~\ref{lem:h-convexity-and-lower-curvature-hyperbolic} below, all
barriers are taken from below at the contact point, exactly as in the
definition above.  Consequently, the inequality
\eqref{eq:barrier-subsolution-definition} is the correct sign for a
maximum-principle argument applied to a function that is about to
violate an upper bound, and no upper barrier convention is needed.
\end{remark}

\begin{lemma}[Strict $h$-convexity and a normalized lower curvature bound]
\label{lem:h-convexity-and-lower-curvature-hyperbolic}
Assume that the conditions of either Theorem~\ref{thm:main} or
Theorem~\ref{thm:main-critical} hold.  Then:
\begin{enumerate}
\item strict $h$-convexity is preserved up to the maximal existence
      time, i.e.
      $$\tau_h=\tau_{\max}.$$
\item if
\[
  \Lambda:=\lambda_{\max}(\check h)
  =\frac{1}{\hat\kappa_1},
\]
then there exists a constant $C>0$, depending only on the coarse
normalized $C^0$ and $C^1$ bounds, the initial normalized speed range,
and the stated radial-weight bounds---in particular, independent of
$\varepsilon_h(M_0)^{-1}$ and of $\tau_{\max}$---such that
\begin{equation}\label{eq:lemma54-Lambda-bound-hyperbolic}
  \Lambda\leq C \text{ on } [0,\tau_{\max}).
\end{equation}
Consequently, all normalized principal curvatures have a uniform
positive lower bound on $[0,\tau_{\max})$.
\end{enumerate}
\end{lemma}

\begin{proof}
The structural condition in either main theorem is equivalent to
\begin{equation}\label{eq:hyperbolic-profile-contact-condition}
  f''(\rho)-\frac{k\alpha}{1+k\alpha}
       \frac{(f'(\rho))^2}{f(\rho)}
  \geq\frac12\tanh\frac\rho2\,f'(\rho).
\end{equation}
\noindent\emph{Step 1. The differential inequality for $\Lambda$ in the barrier sense.}
Since $\check h$ is self-adjoint with respect to $\widetilde g$, its
largest eigenvalue is given by the Rayleigh quotient
\begin{equation}\label{eq:inverse-test-quantity}
  \Lambda(x,\tau):=\lambda_{\max}(\check h)=\sup_{0\neq v\in T_x M}
    \frac{\widetilde g(\check h v,v)}
         {|v|_{\widetilde g}^{2}}.
\end{equation}

Fix an arbitrary point $(x_0,\tau_0)$ and choose normal coordinates
$(x^1,\ldots,x^n)$ for $\widetilde g(\tau_0)$, centred at $x_0$.
Choose the first coordinate direction so that, at $x_0$, $\p_{x^1}$ is
a unit eigenvector of $\check h$ corresponding to
$\Lambda(x_0,\tau_0)$.
Regard the coordinate vector field $\p_{x^1}$ as time-independent,
and define the smooth function
\begin{equation}\label{eq:calabi-support-function}
  \psi(x,\tau)
  :=\frac{
    \widetilde g
    \bigl(\check h(\p_{x^1}),\p_{x^1}\bigr)}{
    \widetilde g(\p_{x^1},\p_{x^1})}.
\end{equation}
By the variational characterization of the largest eigenvalue,
\begin{equation}\label{eq:calabi-support-comparison}
  \psi(x,\tau)\leq\Lambda(x,\tau),
  \qquad
  \psi(x_0,\tau_0)=\Lambda(x_0,\tau_0).
\end{equation}
All the following derivative computations are evaluated at
$(x_0,\tau_0)$.  At this point,
\begin{equation}\label{eq:calabi-contact-data}
  \widetilde g(\p_{x^1},\p_{x^1})=1,
  \qquad
  \check h(\p_{x^1})=\Lambda\p_{x^1},
  \qquad
  \widetilde\nabla\p_{x^1}=0,
  \qquad
  \partial_\tau\p_{x^1}=0.
\end{equation}

For the time derivative,
\begin{equation}\label{eq:calabi-time-derivative-proof}
  \begin{aligned}
    \partial_\tau\psi
    &=(\partial_\tau\widetilde g)
      \bigl(\check h(\p_{x^1}),\p_{x^1}\bigr)
      +\widetilde g
      \bigl((\partial_\tau\check h)(\p_{x^1}),
        \p_{x^1}\bigr)
      -\Lambda(\partial_\tau\widetilde g)
      \bigl(\p_{x^1},\p_{x^1}\bigr)\\
    &=\widetilde g
      \bigl((\partial_\tau\check h)(\p_{x^1}),
        \p_{x^1}\bigr)\\
        & =(\partial_\tau\check h)_1{}^1.
  \end{aligned}
\end{equation}

For the spatial Hessian, metric compatibility and
\eqref{eq:calabi-contact-data} give
\begin{equation*}
  \begin{aligned}
    \widetilde\nabla_p\widetilde\nabla_q
      \Bigl[
        \widetilde g
        \bigl(\check h(\p_{x^1}),\p_{x^1}\bigr)
      \Bigr]=&\widetilde g\bigl(
       (\widetilde\nabla_p\widetilde\nabla_q\check h)
       (\p_{x^1}),\p_{x^1}\bigr)+\widetilde g\bigl(
       \check h(\widetilde\nabla_p\widetilde\nabla_q\p_{x^1}),
       \p_{x^1}\bigr)\\
     &+\widetilde g\bigl(
       \check h(\p_{x^1}),
       \widetilde\nabla_p\widetilde\nabla_q\p_{x^1}\bigr),
       \end{aligned}
       \end{equation*}
       and
\begin{equation*}     
    \widetilde\nabla_p\widetilde\nabla_q
      \Bigl[
        \widetilde g(\p_{x^1},\p_{x^1})
      \Bigr]=\widetilde g\bigl(
       \widetilde\nabla_p\widetilde\nabla_q\p_{x^1},
       \p_{x^1}\bigr)+\widetilde g\bigl(
       \p_{x^1},
       \widetilde\nabla_p\widetilde\nabla_q\p_{x^1}\bigr).
\end{equation*}
Since $\check h$ is diagonal at $(x_0,\tau_0)$, we have
\begin{equation}\label{eq:calabi-cross-term-identity}
  \widetilde g\bigl(
    \check h(\widetilde\nabla_p\widetilde\nabla_q\p_{x^1}),
    \p_{x^1}\bigr)=
    \widetilde g\bigl(
      \check h(\p_{x^1}),
      \widetilde\nabla_p\widetilde\nabla_q\p_{x^1}
    \bigr)=
    \Lambda\widetilde g\bigl(
      \widetilde\nabla_p\widetilde\nabla_q\p_{x^1},
      \p_{x^1}
    \bigr).
\end{equation}
Hence,
\begin{equation}\label{eq:calabi-spatial-hessian-conclusion}
  \begin{aligned}
    \widetilde\nabla_p\widetilde\nabla_q\psi
    ={}&\widetilde\nabla_p\widetilde\nabla_q
      \Bigl[
        \widetilde g
        \bigl(\check h(\p_{x^1}),\p_{x^1}\bigr)
      \Bigr]-\Lambda\widetilde\nabla_p\widetilde\nabla_q
      \Bigl[
        \widetilde g(\p_{x^1},\p_{x^1})
      \Bigr]\\
    ={}&\widetilde g\bigl(
       (\widetilde\nabla_p\widetilde\nabla_q\check h)
       (\p_{x^1}),\p_{x^1}\bigr)\\
     ={}&(\widetilde\nabla_p\widetilde\nabla_q\check h)_1{}^1.
  \end{aligned}
\end{equation}
We have therefore proved, at $(x_0,\tau_0)$, that
\begin{equation}\label{eq:calabi-L-relation}
  \mathcal L\psi
  =(\mathcal L\check h)_1{}^1.
\end{equation}
In the following display, $f$, $f'$, and $f''$ are evaluated at
$\widetilde\rho/\lambda$.  Taking $i=j=1$ in
\eqref{eq:full-inverse-weingarten-evolution-hyperbolic} gives
\begin{equation}\label{eq:inverse-weingarten-11-component}
  \begin{aligned}
  \mathcal L\Lambda
  \leq{}&-\Lambda^2\lambda^\beta f
    \alpha\sigma_k^{\alpha-1}\ddot\sigma_k^{pq,rs}
    \widetilde\nabla_1\hat h_{pq}
    \widetilde\nabla_1\hat h_{rs}\\
  &-\Lambda^2\lambda^\beta f
    \alpha(\alpha-1)\sigma_k^{\alpha-2}
    \dot\sigma_k^{pq}\dot\sigma_k^{rs}
    \widetilde\nabla_1\hat h_{pq}
    \widetilde\nabla_1\hat h_{rs}\\
    &-2\Lambda^2\lambda^\beta f
    \alpha\sigma_k^{\alpha-1}\dot\sigma_k^{rs}
    (\widetilde\nabla_r\hat h_1{}^b)\check h_b{}^p
    (\widetilde\nabla_s\hat h_p{}^1)\\
  &-2\Lambda^2\lambda^{\beta-1}f'
    \alpha\sigma_k^{\alpha-1}\dot\sigma_k^{pq}
    \widetilde\nabla_1\widetilde\rho\,
    \widetilde\nabla_1\hat h_{pq}\\
  &-\Lambda^2\lambda^{\beta-1}f'\sigma_k^\alpha
    \widetilde\nabla_1\widetilde\nabla_1\widetilde\rho-\Lambda^2\lambda^{\beta-2}f''\sigma_k^\alpha
    |\widetilde\nabla_1\widetilde\rho|^2\\
  &-\lambda^\beta f\alpha\sigma_k^{\alpha-1}
    \left(
      \dot\sigma_k^{pq}(\hat h^2)_{pq}
      +\lambda^{-2}\dot\sigma_k^{pq}\widetilde g_{pq}
    \right)\Lambda\\
  &+\lambda^\beta f(k\alpha-1)\sigma_k^\alpha
    +\lambda^{\beta-2}f(1+k\alpha)\sigma_k^\alpha\Lambda^2
    +\gamma_0\Lambda,
  \end{aligned}
\end{equation}
in barrier sense.

\medskip
\noindent\emph{Step 2. The inverse-concavity and radial-derivative estimates.}
Put
\begin{equation}\label{eq:F-and-G}
  G:=\sigma_k^{1/k}(\hat h),
\end{equation}
then Lemma~\ref{lem:sigma-k-root-inverse-concavity}, applied with
$A=\hat h$ and $A^{-1}=\check h$, yields
\begin{equation}\label{eq:inverse-concavity-inequality}
    \ddot G^{pq,rs}\eta_{pq}\eta_{rs}
    +2\dot G^{ij}\check h^{pq}\eta_{ip}\eta_{jq}
    \geq
    2G^{-1}\bigl(\dot G^{pq}\eta_{pq}\bigr)^2.
\end{equation}
Using $F=G^{k\alpha}$, we derive
\begin{equation}\label{eq:inverse-concavity-for-F}
  \begin{aligned}
  &\ddot F^{pq,rs}\eta_{pq}\eta_{rs}
   +2\dot F^{ij}\check h^{pq}\eta_{ip}\eta_{jq}\\
  ={}&k\alpha G^{k\alpha-1}
    \left(
      \ddot G^{pq,rs}\eta_{pq}\eta_{rs}
      +2\dot G^{ij}\check h^{pq}\eta_{ip}\eta_{jq}
    \right)
    +k\alpha(k\alpha-1)G^{k\alpha-2}
      \bigl(\dot G^{pq}\eta_{pq}\bigr)^2\\
  \geq{}&k\alpha(k\alpha+1)G^{k\alpha-2}
      \bigl(\dot G^{pq}\eta_{pq}\bigr)^2.
  \end{aligned}
\end{equation}
Set
$\eta_{pq}=\widetilde\nabla_1\hat h_{pq}$. Then the first three gradient terms in
\eqref{eq:inverse-weingarten-11-component} are exactly the negative
factor $-\Lambda^2\lambda^\beta f(\ddot F^{pq,rs}\eta_{pq}\eta_{rs}
   +2\dot F^{ij}\check h^{pq}\eta_{ip}\eta_{jq})$.  Since $f>0$ and
$\dot G^{pq}\eta_{pq}=\widetilde\nabla_1G$, we conclude that
\begin{equation}\label{eq:inverse-gradient-quadratic-combination}
  \begin{aligned}
  &-\Lambda^2
  \lambda^\beta f
  \Bigl(
    \ddot F^{pq,rs}\eta_{pq}\eta_{rs}
    +2\dot F^{ij}\check h^{pq}\eta_{ip}\eta_{jq}
  \Bigr)\\
  &\quad\leq
  -\Lambda^2
  \lambda^\beta f
  k\alpha(k\alpha+1)G^{k\alpha-2}
  |\widetilde\nabla_1G|^2.
  \end{aligned}
\end{equation}
Here and below $f$, $f'$, and $f''$ are evaluated at
$\widetilde\rho/\lambda$.  The mixed radial term in
\eqref{eq:inverse-weingarten-11-component} becomes
\begin{equation}\label{eq:mixed-radial-gradient-term}
  -2\Lambda^2\lambda^{\beta-1}f'
    \alpha\sigma_k^{\alpha-1}\dot\sigma_k^{pq}
    \widetilde\nabla_1\widetilde\rho\,
    \widetilde\nabla_1\hat h_{pq} =-2\Lambda^2\lambda^{\beta-1}f'
  \widetilde\nabla_1\widetilde\rho\,
  \widetilde\nabla_1F
  =-2\Lambda^2\lambda^{\beta-1}f'
  k\alpha G^{k\alpha-1}
  \widetilde\nabla_1\widetilde\rho\,
  \widetilde\nabla_1G.
\end{equation}
Combining \eqref{eq:inverse-gradient-quadratic-combination} and
\eqref{eq:mixed-radial-gradient-term}, and completing the square, yields
\begin{equation}\label{eq:completed-square-gradient-bound}
  \begin{aligned}
  &-\Lambda^2k\alpha\lambda^\beta fG^{k\alpha-2}
  \Bigg[
    (k\alpha+1)|\widetilde\nabla_1G|^2
    +2\frac{f'}{\lambda f}G
      \widetilde\nabla_1\widetilde\rho\,
      \widetilde\nabla_1G
  \Bigg]\\
  &\quad\leq
  \Lambda^2\frac{k\alpha}{k\alpha+1}
  \lambda^{\beta-2}\frac{(f')^2}{f}
  F|\widetilde\nabla_1\widetilde\rho|^2.
  \end{aligned}
\end{equation}
Recall
\begin{equation}\label{eq:unscaled-radial-hessian}
  \nabla_i\nabla_j\rho
  =\coth\rho\,
    \bigl(
      g_{ij}-\nabla_i\rho\,\nabla_j\rho
    \bigr)
   -v^{-1}h_{ij}.
\end{equation}
Because $\lambda$ is spatially constant, $\widetilde\nabla=\nabla$ and
\begin{equation*}
  \widetilde\nabla_i\widetilde\nabla_j\widetilde\rho
  =\lambda\nabla_i\nabla_j\rho.
\end{equation*}
Using
\begin{equation*}
  \rho=\frac{\widetilde\rho}{\lambda},
  \qquad
  g_{ij}=\lambda^{-2}\widetilde g_{ij},
  \qquad
  \nabla_i\rho
    =\lambda^{-1}\widetilde\nabla_i\widetilde\rho,
  \qquad
  h_{ij}=\lambda^{-1}\hat h_{ij},
\end{equation*}
we obtain from \eqref{eq:unscaled-radial-hessian}
\begin{equation}\label{eq:scaled-radial-hessian}
  \widetilde\nabla_i\widetilde\nabla_j\widetilde\rho
  =\lambda^{-1}\coth
    \left(\frac{\widetilde\rho}{\lambda}\right)
    \left(
      \widetilde g_{ij}
      -\widetilde\nabla_i\widetilde\rho\,
       \widetilde\nabla_j\widetilde\rho
    \right)
   -v^{-1}\hat h_{ij}.
\end{equation}
Consequently,
\begin{equation}\label{eq:pure-radial-terms-at-maximum}
  \begin{aligned}
  &-\Lambda^2\lambda^{\beta-1}f'\sigma_k^\alpha
    \widetilde\nabla_1\widetilde\nabla_1\widetilde\rho-\Lambda^2\lambda^{\beta-2}f''\sigma_k^\alpha
    |\widetilde\nabla_1\widetilde\rho|^2\\
    =&-\Lambda^2F\Bigg[
   \lambda^{\beta-2}f''
     |\widetilde\nabla_1\widetilde\rho|^2
   +\lambda^{\beta-2}f'\coth
     \left(\frac{\widetilde\rho}{\lambda}\right)
     \left(1-|\widetilde\nabla_1\widetilde\rho|^2\right)
  \Bigg]
  +\Lambda\lambda^{\beta-1}f'v^{-1}F.
  \end{aligned}
\end{equation}
Set
\[
  r:=\frac{\widetilde\rho}{\lambda}=\rho.
\]
By Euler's identity \eqref{eq:lem25-euler-identity-hyperbolic} and
$F=\sigma_k^\alpha$, the zero-order curvature terms satisfy
\begin{equation}\label{eq:euler-zero-order-combination}
\begin{aligned}
  &-\lambda^\beta f(r)\alpha\sigma_k^{\alpha-1}
    \left(
      \dot\sigma_k^{pq}(\hat h^2)_{pq}
      +\lambda^{-2}\dot\sigma_k^{pq}\widetilde g_{pq}
    \right)\Lambda\\
  &\quad+\lambda^\beta f(r)(k\alpha-1)F
    +\lambda^{\beta-2}f(r)(1+k\alpha)F\Lambda^2\\
  ={}&-\lambda^\beta f(r)\alpha\sigma_k^{\alpha-1}\Lambda
    \sum_i\sigma_{k-1}(\hat\kappa\mid i)
       (\hat\kappa_i-\lambda^{-1})^2\\
  &+\lambda^{\beta-2}f(r)F
    \left[
      (1+k\alpha)\Lambda^2-2k\alpha\lambda\Lambda
      +(k\alpha-1)\lambda^2
    \right]\\
  ={}&-\lambda^\beta f(r)\alpha\sigma_k^{\alpha-1}\Lambda
    \sum_i\sigma_{k-1}(\hat\kappa\mid i)
       (\hat\kappa_i-\lambda^{-1})^2\\
  &+\lambda^{\beta-2}f(r)F(\Lambda-\lambda)
    \bigl((1+k\alpha)\Lambda+(1-k\alpha)\lambda\bigr).
\end{aligned}
\end{equation}
Combining \eqref{eq:euler-zero-order-combination} with the
inverse-concavity estimate
\eqref{eq:inverse-gradient-quadratic-combination}, the completed-square
bound \eqref{eq:completed-square-gradient-bound}, and the radial
identity \eqref{eq:pure-radial-terms-at-maximum} in
\eqref{eq:inverse-weingarten-11-component}, we obtain, in the barrier
sense,
\begin{equation}\label{eq:Lambda-after-derivative-estimates}
\begin{aligned}
  \mathcal L\Lambda
  \leq{}&-\Lambda^2\lambda^{\beta-2}F
   \left[
     f'(r)\coth r\,
       \bigl(1-|\widetilde\nabla_1\widetilde\rho|^2\bigr)
     +\left(
       f''(r)-\frac{k\alpha}{1+k\alpha}
          \frac{(f'(r))^2}{f(r)}
      \right)|\widetilde\nabla_1\widetilde\rho|^2
   \right]\\
  &+\Lambda\lambda^{\beta-1}f'(r)
      \l\p_\rho,\nu\r F\\
  &-\lambda^\beta f(r)\alpha\sigma_k^{\alpha-1}\Lambda
    \sum_i\sigma_{k-1}(\hat\kappa\mid i)
       (\hat\kappa_i-\lambda^{-1})^2\\
  &+\lambda^{\beta-2}f(r)F(\Lambda-\lambda)
     \bigl((1+k\alpha)\Lambda+(1-k\alpha)\lambda\bigr)
   +\gamma_0\Lambda.
\end{aligned}
\end{equation}

\medskip
\noindent\emph{Step 3. Preservation of strict $h$-convexity and the maximum-principle estimate.}
We now use the preceding calculation before making any quantitative
estimate for $\Lambda$.  Set
\[
  W:=\lambda^{-1}\Lambda=\frac{1}{\kappa_1}.
\]
Since $\partial_\tau\lambda=\gamma_0\lambda$ and $\lambda$ is spatially
constant,
\begin{equation}\label{eq:W-operator-from-Lambda}
\begin{aligned}
  \mathcal LW
  ={}&\lambda^{-1}
     \bigl(\mathcal L\Lambda-\gamma_0\Lambda\bigr)\\
  \leq{}&-\Lambda^2\lambda^{\beta-3}F
   \left[
     f'(r)\coth r\,
       \bigl(1-|\widetilde\nabla_1\widetilde\rho|^2\bigr)
     +\left(
       f''(r)-\frac{k\alpha}{1+k\alpha}
          \frac{(f'(r))^2}{f(r)}
      \right)|\widetilde\nabla_1\widetilde\rho|^2
   \right]\\
  &+\Lambda\lambda^{\beta-2}f'(r)
      \l\p_\rho,\nu\r F\\
  &-\lambda^{\beta-1}f(r)\alpha\sigma_k^{\alpha-1}\Lambda
    \sum_i\sigma_{k-1}(\hat\kappa\mid i)
       (\hat\kappa_i-\lambda^{-1})^2\\
  &+\lambda^{\beta-3}f(r)F(\Lambda-\lambda)
     \bigl((1+k\alpha)\Lambda+(1-k\alpha)\lambda\bigr).
\end{aligned}
\end{equation}
Suppose that $\tau_h<\tau_{\max}$ and that strict $h$-convexity is
lost at a first point
$(x_0,\tau_0)\in M\times(0,\tau_h]$, i.e.\ $W(x_0,\tau_0)=1$,
equivalently $\Lambda=\lambda$; by the definition of $\tau_h$ one
necessarily has $\tau_0=\tau_h$.  Since
$\tau_0=\tau_h<\tau_{\max}$, the solution is smooth up to $\tau_0$,
and the normalized $C^0$ estimate
\eqref{eq:normalized-C0-estimate-hyperbolic} gives
\[
  \rho\bigl(\cdot,t(\tau_0)\bigr)
  \geq c_0\,\lambda\bigl(t(\tau_0)\bigr)^{-1}>0,
\]
so $o$ remains in the interior of the domain enclosed by
$M_{t(\tau_0)}$, and the strict conclusion of
Lemma~\ref{lem:global-supporting-horosphere} is available at the
contact point.
At this point, \eqref{eq:W-operator-from-Lambda} implies
\begin{equation}\label{eq:W-first-contact-inequality}
\begin{aligned}
  \mathcal LW
  \leq{}&-\lambda^{\beta-1}F
  \Bigg[
    f'(\rho)\Bigl(
      \coth\rho\bigl(1-|\widetilde\nabla_1\widetilde\rho|^2\bigr)
      -\l\p_\rho,\nu\r
    \Bigr)
    +\left(
       f''(\rho)-\frac{k\alpha}{1+k\alpha}
          \frac{(f'(\rho))^2}{f(\rho)}
     \right)|\widetilde\nabla_1\widetilde\rho|^2
  \Bigg].
\end{aligned}
\end{equation}
\medskip
\noindent\emph{Claim 1.} At the contact point,
\begin{equation}\label{eq:contact-radial-gradient-bound}
  0\leq|\widetilde\nabla_1\widetilde\rho|^2
  \leq1-\l\p_\rho,\nu\r^2.
\end{equation}

\begin{proof}[Proof of Claim 1]
Decompose the ambient radial vector into its tangential and normal parts:
\[
  \partial_\rho
  =\nabla^M\rho+\l\p_\rho,\nu\r\nu.
\]
Since $|\partial_\rho|=|\nu|=1$, this gives
\begin{equation}\label{eq:radial-gradient-normal-decomposition}
  |\nabla^M\rho|_g^2
  =1-\l\p_\rho,\nu\r^2.
\end{equation}
If $\widetilde e_1$ is the $\widetilde g$-unit contact direction, then
$e_1:=\lambda\widetilde e_1$ is $g$-unit and, because
$\widetilde\rho=\lambda\rho$ with spatially constant $\lambda$,
\[
  \widetilde e_1(\widetilde\rho)=e_1(\rho).
\]
Therefore
\[
  |\widetilde\nabla_1\widetilde\rho|^2
   =|e_1(\rho)|^2
   \leq|\nabla^M\rho|_g^2
   =1-\l\p_\rho,\nu\r^2,
\]
which proves the claim.
\end{proof}

\medskip
\noindent\emph{Claim 2.} The square-bracketed term in
\eqref{eq:W-first-contact-inequality} is strictly positive.

\begin{proof}[Proof of Claim 2]
By Lemma~\ref{lem:global-supporting-horosphere} and the fact that
$\partial_\rho$ and $\nu$ are unit vectors,
\[
  \tanh\frac{\rho}{2}<\l\p_\rho,\nu\r\leq1.
\]
We first consider the case $\l\p_\rho,\nu\r<1$.
The identity
\[
  2\tanh\frac{\rho}{2}\coth\rho
  =1+\tanh^2\frac{\rho}{2}
\]
gives
\[
  \frac{
    \l\p_\rho,\nu\r
    -\coth\rho\,\l\p_\rho,\nu\r^2
  }{
    1-\l\p_\rho,\nu\r^2
  }
  -\frac12\tanh\frac{\rho}{2}
  =-\frac{
    \bigl(\l\p_\rho,\nu\r-\tanh(\rho/2)\bigr)^2
  }{
    2\tanh(\rho/2)\bigl(1-\l\p_\rho,\nu\r^2\bigr)
  }<0.
\]
Hence
\begin{equation}\label{eq:hyperbolic-contact-coefficient-inequality}
  \frac{
    \l\p_\rho,\nu\r
    -\coth\rho\,\l\p_\rho,\nu\r^2
  }{
    1-\l\p_\rho,\nu\r^2
  }
  <\frac12\tanh\frac{\rho}{2}.
\end{equation}
View the square-bracketed term as an affine function of
$s:=|\widetilde\nabla_1\widetilde\rho|^2$; its minimum on the
interval $[0,1-\l\p_\rho,\nu\r^2]$ is attained at one of the
endpoints.  It is therefore enough to check both endpoints.  At $s=0$ the value is
\[
  f'(\rho)\bigl(\coth\rho-\l\p_\rho,\nu\r\bigr)>0.
\]
Indeed, both $\partial_\rho$ and $\nu$ are unit vectors, so
$\l\p_\rho,\nu\r\leq1$, whereas $\rho>0$ implies
$\coth\rho>1$.  Hence
$\coth\rho-\l\p_\rho,\nu\r>0$, and the strict inequality above follows
from $f'(\rho)>0$.
At $s=1-\l\p_\rho,\nu\r^2$ the value is
\[
\begin{aligned}
 &f'(\rho)\Bigl(
   \l\p_\rho,\nu\r^2\coth\rho-\l\p_\rho,\nu\r
 \Bigr)
 +\bigl(1-\l\p_\rho,\nu\r^2\bigr)
 \left(
   f''(\rho)-\frac{k\alpha}{1+k\alpha}
      \frac{(f'(\rho))^2}{f(\rho)}
 \right)\\
 &\quad=\bigl(1-\l\p_\rho,\nu\r^2\bigr)\left[
   f''(\rho)-\frac{k\alpha}{1+k\alpha}
      \frac{(f'(\rho))^2}{f(\rho)}
   -f'(\rho)
    \frac{
      \l\p_\rho,\nu\r
      -\coth\rho\,\l\p_\rho,\nu\r^2
    }{
      1-\l\p_\rho,\nu\r^2
    }
 \right]\\
 &\quad>\bigl(1-\l\p_\rho,\nu\r^2\bigr)\left[
   f''(\rho)-\frac{k\alpha}{1+k\alpha}
      \frac{(f'(\rho))^2}{f(\rho)}
   -\frac12\tanh\frac\rho2\,f'(\rho)
 \right]\geq0
\end{aligned}
\]
Here the strict inequality follows from
\eqref{eq:hyperbolic-contact-coefficient-inequality} and
$f'(\rho)>0$, while the final inequality follows from
\eqref{eq:hyperbolic-profile-contact-condition}.  If
$\l\p_\rho,\nu\r=1$, then
$|\widetilde\nabla_1\widetilde\rho|^2=0$, and the square-bracketed term
equals
\[
  f'(\rho)(\coth\rho-1)>0.
\]
This proves the claim.
\end{proof}
It follows from Claim~2 and \eqref{eq:W-first-contact-inequality} that
$\mathcal LW<0$ at the first contact, contradicting the parabolic
first-contact inequality.  Thus strict $h$-convexity is preserved.

\medskip
\noindent\emph{Step 4. The quantitative lower curvature bound.}
We start directly from \eqref{eq:Lambda-after-derivative-estimates}.
By Step~3, strict $h$-convexity is preserved, and hence
$0<\Lambda\leq\lambda$.  The coefficient bounds and
$|\l\p_\rho,\nu\r|\leq1$ give
\[
  \Lambda\lambda^{\beta-1}f'(r)\l\p_\rho,\nu\r F
  =\widetilde\Phi\Lambda\frac{f'(r)}{\lambda f(r)}
     \l\p_\rho,\nu\r
  \leq C\widetilde\Phi\Lambda.
\]
Moreover, set
\[
  x:=\frac{\Lambda}{\lambda}\in(0,1],
  \qquad m:=k\alpha>0.
\]
The polynomial
\[
  (x-1)\bigl((1+m)x+1-m\bigr)
\]
is bounded above on $[0,1]$ by a constant depending only on $m$.
Consequently,
\[
\begin{aligned}
  &\lambda^{\beta-2}f(r)F(\Lambda-\lambda)
     \bigl((1+k\alpha)\Lambda+(1-k\alpha)\lambda\bigr)\\
  &\qquad
  =\widetilde\Phi
     \left(\frac{\Lambda}{\lambda}-1\right)
     \left((1+k\alpha)\frac{\Lambda}{\lambda}+1-k\alpha\right)
  \leq C\widetilde\Phi.
\end{aligned}
\]
The square term in \eqref{eq:Lambda-after-derivative-estimates} is
nonpositive and may be discarded, while $\gamma_0\Lambda\leq C\Lambda$.
Thus \eqref{eq:Lambda-after-derivative-estimates} directly yields
\begin{equation}\label{eq:Lambda-key-coefficient}
  \begin{aligned}
  \mathcal L\Lambda
  \leq{}&C\widetilde\Phi\Lambda+C\widetilde\Phi+C\Lambda\\
  &-\Lambda^2\lambda^{\beta-2}F\Bigg[
    f'(r)\coth r\,
      \bigl(1-|\widetilde\nabla_1\widetilde\rho|^2\bigr)\\
  &\hspace{26mm}
    +\left(
      f''(r)-\frac{k\alpha}{1+k\alpha}
        \frac{(f'(r))^2}{f(r)}
     \right)|\widetilde\nabla_1\widetilde\rho|^2
  \Bigg].
  \end{aligned}
\end{equation}

The gradient estimate gives
\begin{equation}\label{eq:radial-gradient-good-factor}
  1-|\widetilde\nabla_1\widetilde\rho|^2
  \geq1-|\widetilde\nabla^M\widetilde\rho|^2
  =v^{-2}\geq c_0>0.
\end{equation}

\medskip
\noindent\emph{Coercivity and the maximum principle.}
Choose $R\in(0,\infty)$ so that
$0<r=\widetilde\rho/\lambda\leq R$ on $[0,\tau_h)$.  Lemmas
\ref{lem:model-profile-comparison-hyperbolic},
\ref{lem:profile-logarithmic-derivative-bound-hyperbolic}, and
\ref{lem:profile-second-derivative-bound-hyperbolic} imply that, for
every $r\in(0,R]$,
\begin{equation}\label{eq:general-f-structural-growth}
  \begin{aligned}
  c_f r^\beta&\leq f(r)\leq C_f r^\beta,\\
  c_f r^{\beta-1}&\leq f'(r)\leq C_f r^{\beta-1},\\
  |f''(r)|&\leq C_f r^{\beta-2},
  \end{aligned}
\end{equation}
The root-convexity assumption in the corresponding theorem gives
\begin{equation}\label{eq:general-f-root-convexity}
  \left(f^{\frac{1}{1+k\alpha}}\right)''
  \geq\frac12\tanh\frac r2
  \left(f^{\frac{1}{1+k\alpha}}\right)'\geq0.
\end{equation}
The latter condition is equivalent to
\begin{equation}\label{eq:general-f-inverse-concavity-combination}
  f''-\frac{k\alpha}{1+k\alpha}\frac{(f')^2}{f}
  =(1+k\alpha)f^{\frac{k\alpha}{1+k\alpha}}
    \left(f^{\frac{1}{1+k\alpha}}\right)''
  \geq0.
\end{equation}

By \eqref{eq:radial-gradient-good-factor}, there is a fixed
$\vartheta>0$ such that
\[
  1-|\widetilde\nabla_1\widetilde\rho|^2\geq\vartheta.
\]
Therefore the bracket in
\eqref{eq:Lambda-key-coefficient} satisfies
\begin{equation}\label{eq:general-f-uniform-bracket}
  \begin{aligned}
  &\lambda^{\beta-2}\left[
    f'(r)\coth r\,
      \bigl(1-|\widetilde\nabla_1\widetilde\rho|^2\bigr)
    +\left(
      f''-\frac{k\alpha}{1+k\alpha}\frac{(f')^2}{f}
     \right)|\widetilde\nabla_1\widetilde\rho|^2
  \right]\\
  &\quad\geq
  \vartheta\lambda^{\beta-2}f'(r)\coth r\\
  &\quad\geq
  c\lambda^{\beta-2}r^{\beta-1}\coth r\\
  &\quad\geq
  c\lambda^{\beta-2}r^{\beta-2}
  =c\widetilde\rho^{\beta-2}
  \geq c.
  \end{aligned}
\end{equation}
Here we used the positive lower bound
$r\coth r\geq1$ on $(0,R]$ and the two-sided $C^0$ bound for
$\widetilde\rho$.  Since
\[
  \widetilde\Phi=\lambda^\beta f(r)F
  \quad\text{and}\quad
  \lambda^\beta f(r)\leq C,
\]
the negative quadratic term in
\eqref{eq:Lambda-key-coefficient} is bounded above by
$-c\widetilde\Phi\Lambda^2$.  We therefore obtain
\begin{equation}\label{eq:Lambda-general-profile-all-time-inequality}
  \mathcal L\Lambda
  \leq-c\widetilde\Phi\Lambda^2
       +C\widetilde\Phi\Lambda
       +C\widetilde\Phi+C\Lambda.
\end{equation}
This is the point at which the positive lower speed bound, rather than
an upper speed bound, is used.  By
\eqref{eq:normalized-speed-lower-bound-hyperbolic},
\[
  C\Lambda
  \leq\frac{C}{c_{\widetilde\Phi}}
       \widetilde\Phi\Lambda.
\]
After changing $C$, \eqref{eq:Lambda-general-profile-all-time-inequality}
becomes
\begin{equation}\label{eq:Lambda-factorized-by-speed-hyperbolic}
  \mathcal L\Lambda
  \leq\widetilde\Phi
       \left(-c\Lambda^2+C\Lambda+C\right).
\end{equation}
Choose $R_*>1$, depending only on the constants $c,C$ in
\eqref{eq:Lambda-factorized-by-speed-hyperbolic}, so that
\[
  -cs^2+Cs+C<0\qquad\text{for every }s\geq R_*.
\]
No upper bound for $\widetilde\Phi$ is needed here, since
$\widetilde\Phi>0$.  The barrier maximum principle gives the estimate
\begin{equation}\label{eq:Lambda-maximum-principle-bound}
  \sup_{\mathbb S^n\times[0,\tau_h)}\Lambda
  \leq
  \max\left\{\sup_{\mathbb S^n}\Lambda(\cdot,0),R_*\right\}
  \leq R_*,
\end{equation}
where the last inequality uses
\[
  \sup_{\mathbb S^n}\Lambda(\cdot,0)
  =\frac{1}{\min_{M_0}\kappa_1(\cdot,0)}
  <1<R_*,
\]
since $\min_{M_0}\kappa_1(\cdot,0)\geq1+\varepsilon_h(M_0)>1$.
Consequently,
\begin{equation}\label{eq:minimum-curvature-bound}
  \Lambda\leq R_*,
  \qquad
  \hat\kappa_i\geq R_*^{-1}
  \qquad\text{on }\mathbb S^n\times[0,\tau_h),
\end{equation}
where $R_*$ is independent of $\tau_h$ and of the strict
$h$-convexity margin $\varepsilon_h(M_0)$ of the initial
hypersurface.

Suppose that $\tau_h<\tau_{\max}$.  Since the solution is smooth at
$\tau_h$, compactness, continuity, and the maximality of $\tau_h$
imply that at $\tau_h$ either strict $h$-convexity fails,
\[
  \min_{M_{t(\tau_h)}}\min_{1\leq i\leq n}\kappa_i=1,
\]
or enclosure fails, $o\in\partial\Omega_{\tau_h}$.  The second
alternative is impossible: the normalized $C^0$ estimate
\eqref{eq:normalized-C0-estimate-hyperbolic} gives
\[
  \min_{\mathbb S^n}\rho\bigl(\cdot,t(\tau_h)\bigr)
  \geq c_0\,\lambda\bigl(t(\tau_h)\bigr)^{-1}
  =c_0e^{-\gamma_0\tau_h}>0,
\]
so $o$ remains in the
interior of the domain enclosed by $M_{t(\tau_h)}$.  Therefore the
first alternative holds, and the
first-contact argument in Step~3 applies at time $\tau_h$ and gives a
contradiction.  Hence $\tau_h=\tau_{\max}$, and
\eqref{eq:minimum-curvature-bound} gives the desired positive lower
bound for the normalized principal curvatures throughout the maximal
existence interval $[0,\tau_{\max})$.
\end{proof}

\begin{lemma}[Upper bound for the normalized speed]
\label{lem:normalized-speed-upper-bound-hyperbolic}
Assume that the conditions of either Theorem~\ref{thm:main} or
Theorem~\ref{thm:main-critical} hold.  Then there exists a constant
$C>0$, independent of $\varepsilon_h(M_0)^{-1}$ and of
$\tau_{\max}$, such that
\begin{equation}\label{eq:C2-normalized-speed-upper-bound-hyperbolic}
  \widetilde\Phi\leq C
  \qquad\text{on }\mathbb S^n\times[0,\tau_{\max}).
\end{equation}
\end{lemma}

\begin{proof}
By Lemma~\ref{lem:h-convexity-and-lower-curvature-hyperbolic}, the
solution remains strictly $h$-convex throughout its maximal existence
interval $[0,\tau_{\max})$.  The $C^1$ estimate gives two-sided positive bounds for the
normalized support function $\widetilde u$.  Set
\begin{equation}\label{eq:C2-speed-test-function-hyperbolic}
  b:=\frac12\inf_{\mathbb S^n\times[0,\tau_{\max})}\widetilde u,
  \qquad
  z:=\widetilde u-b,
  \qquad
  Q:=\log\widetilde\Phi-\log z.
\end{equation}
Thus $z$ is bounded above and below by positive constants.

At a spatial maximum of $Q$, one has
$\widetilde\nabla\log\widetilde\Phi=\widetilde\nabla\log z$;
hence the gradient-square terms in $\mathcal LQ$ cancel.  Using
\eqref{eq:basic-normalized-speed-evolution-hyperbolic} and
\eqref{eq:basic-normalized-support-evolution-hyperbolic}, we obtain
\begin{equation}\label{eq:C2-speed-quotient-evolution-hyperbolic}
  \begin{aligned}
  \mathcal LQ
  ={}&(\beta-k\alpha)\gamma_0
       -\gamma_0\frac{\widetilde u}{z}
       -\frac{f'(\rho)}{\lambda f(\rho)v}\widetilde\Phi
       +\frac{(1+k\alpha)\cosh\rho}{z}\widetilde\Phi\\
    &-\frac{b}{z}\lambda^\beta f(\rho)
       \dot F^{ij}(\hat h^2)_{ij}
       -\lambda^{\beta-2}f(\rho)
       \dot F^{ij}\widetilde g_{ij}\\
    &-\frac{\lambda^\beta f'(\rho)F\sinh\rho}{z}
       (1-v^{-2}).
  \end{aligned}
\end{equation}
Since $f'>0$ and $v\geq1$, the third and the last terms on the
right-hand side are nonpositive.  The term containing
$\dot F^{ij}\widetilde g_{ij}$ is also nonpositive.  The remaining
coefficients are uniformly bounded, and hence
\begin{equation}\label{eq:C2-speed-quotient-upper-bound-hyperbolic}
  \mathcal LQ
  \leq C+C\widetilde\Phi
       -\frac{b}{z}\lambda^\beta f(\rho)
        \dot F^{ij}(\hat h^2)_{ij}.
\end{equation}

The Newton--Maclaurin inequalities, the homogeneity of
$F=\sigma_k^\alpha$, and
\eqref{eq:normalized-profile-coefficient-bounds} imply
\begin{equation}\label{eq:C2-speed-coercive-term-hyperbolic}
  \lambda^\beta f(\rho)
  \dot F^{ij}(\hat h^2)_{ij}
  \geq c\widetilde\Phi^{1+\frac1{k\alpha}}.
\end{equation}
It follows that, at a spatial maximum of $Q$,
\[
  \mathcal LQ
  \leq C+C\widetilde\Phi
       -c\widetilde\Phi^{1+\frac1{k\alpha}}.
\]
The right-hand side is negative whenever $\widetilde\Phi$ is sufficiently
large.  Since $z$ has two-sided positive bounds, the parabolic maximum
principle applied to $Q$ gives
\[
  \widetilde\Phi\leq C
  \qquad\text{on }\mathbb S^n\times[0,\tau_{\max}).
\]
This proves \eqref{eq:C2-normalized-speed-upper-bound-hyperbolic}.
\end{proof}

\begin{lemma}[The normalized $C^2$ estimate]
\label{lem:normalized-C2-estimate-hyperbolic}
Assume that the conditions of either Theorem~\ref{thm:main} or
Theorem~\ref{thm:main-critical} hold. Then there exists a constant
$C>0$, independent of $\varepsilon_h(M_0)^{-1}$ and of
$\tau_{\max}$, such that
\begin{equation}\label{eq:lemma54-full-C2-conclusions-hyperbolic}
  \widetilde\Phi\leq C,
  \qquad
  C^{-1}\leq\hat\kappa_i\leq C,
  \qquad
  |\nabla_\sigma^2\widetilde\rho|_\sigma\leq C.
\end{equation}
These estimates hold throughout the maximal existence interval
$[0,\tau_{\max})$.
\end{lemma}

\begin{proof}
By Lemmas~\ref{lem:h-convexity-and-lower-curvature-hyperbolic} and
\ref{lem:normalized-speed-upper-bound-hyperbolic},
\begin{equation}\label{eq:C2-speed-and-lower-curvature-hyperbolic}
  \widetilde\Phi\leq C,
  \qquad
  \hat\kappa_i\geq c>0.
\end{equation}
Since
$\widetilde\Phi=\lambda^\beta f(\rho)\sigma_k^\alpha(\hat h)$
and \eqref{eq:normalized-profile-coefficient-bounds} gives two-sided
positive bounds for $\lambda^\beta f(\rho)$, the speed bound implies
\[
  \sigma_k(\hat h)\leq C.
\]
On the other hand, the lower curvature bound in
\eqref{eq:C2-speed-and-lower-curvature-hyperbolic} gives, for every $i$,
\[
  \sigma_k(\hat h)\geq c_1\hat\kappa_i,
\]
where $c_1>0$ depends only on $n$, $k$, and the lower curvature bound.
Consequently,
\[
  C^{-1}\leq\hat\kappa_i\leq C.
\]

It remains to control the Hessian on $(\mathbb S^n,\sigma)$.  Since
$\rho=\widetilde\rho/\lambda$ and
$\hat h_i{}^j=\lambda^{-1}h_i{}^j$, the radial graph formula
\eqref{eq:weingarten-map-rho} is equivalent to
\begin{equation}\label{eq:C2-hyperbolic-Hessian-from-Weingarten}
  \begin{aligned}
  \left(
    \sigma^{jp}
    -\frac{\widetilde\rho^j\widetilde\rho^p}
      {\lambda^2\sinh^2\rho\,v^2}
  \right)\widetilde\rho_{pi}
  ={}&\lambda\sinh\rho\cosh\rho\,\delta_i{}^j\\
  &+\frac{\coth\rho}{\lambda v^2}
      \widetilde\rho_i\widetilde\rho^j
    -\lambda^2v\sinh^2\rho\,\hat h_i{}^j.
  \end{aligned}
\end{equation}
Here all derivatives are taken with respect to $\sigma$.  The
coefficient matrix on the left-hand side has eigenvalues $1$ in the
directions orthogonal to $\nabla\widetilde\rho$ and $v^{-2}$ in its
gradient direction.  It and its inverse are therefore uniformly
bounded by the $C^1$ estimate.  The $C^0$ estimate and the two-sided
curvature bound control every term on the right-hand side.  Hence
\[
  |\nabla_\sigma^2\widetilde\rho|_\sigma\leq C.
\]
This proves \eqref{eq:lemma54-full-C2-conclusions-hyperbolic}.
\end{proof}

\begin{proposition}[Higher regularity and continuation]
\label{prop:C2-to-Cinfty-normalized-hyperbolic}
Let $\widetilde\rho$ be a smooth solution of
\eqref{eq:normalized-radial-flow-hat-h} on its maximal normalized time
interval $[0,\tau_{\max})$, and suppose that the conclusions of the
normalized $C^0$, $C^1$, and $C^2$ estimates hold.  Then
$\tau_{\max}=\infty$.  If, in
addition, the radial weight satisfies
\eqref{eq:profile-symbol-bounds-all-orders}, then, for every integer
$m\geq0$, there exists a constant $C_m>0$ such that
\begin{equation}\label{eq:all-higher-order-estimates-hyperbolic}
  \|\widetilde\rho(\cdot,\tau)\|_{C^m(\mathbb S^n)}
  \leq C_m,
  \qquad \tau\geq0.
\end{equation}
\end{proposition}

\begin{proof}
Set $u=\widetilde\rho$, $\rho=u/\lambda$, and
$\mathcal S=\sigma_k^{1/k}$.  In any fixed coordinate chart on
$\mathbb S^n$, the normalized equation has the form
\begin{equation}\label{eq:normalized-scalar-G-equation}
  u_\tau
  =G(\theta,\tau,u,Du,D^2u).
\end{equation}
For fixed $(\theta,\tau,z,p)$, the radial graph formula expresses the
normalized Weingarten map, after conjugation by a positive definite
matrix depending only on $(\theta,\tau,z,p)$, as an affine function
$\mathcal H(r)$ of $r=D^2u$.  Thus
\[
  G(\theta,\tau,z,p,r)
  =\gamma_0z
   -\mathcal A(\theta,\tau,z,p)
       \mathcal S(\mathcal H(r))^{k\alpha},
  \qquad
  \mathcal A
  =\lambda^\beta f(z/\lambda)
    \left(
      1+\frac{|p|_\sigma^2}
       {\lambda^2\sinh^2(z/\lambda)}
    \right)^{1/2}.
\]
The normalized $C^0$, $C^1$, and $C^2$ estimates imply that
$\mathcal A$ is bounded above and below by positive constants, that the
eigenvalues of $\mathcal H=\hat h$ remain in a fixed compact subset of
$\Gamma_+$, and that all conjugating matrices and their inverses are
uniformly bounded.  Consequently, the linearization of
\eqref{eq:normalized-scalar-G-equation} is uniformly parabolic.

Although $\sigma_k^\alpha$ need not be concave, the required level-set
condition is preserved under the positive power.  Indeed, for fixed
$(\theta,\tau,z,p)$, every nontrivial sublevel set of
$r\mapsto G(\theta,\tau,z,p,r)$ is of the form
\[
  \{r:G(\theta,\tau,z,p,r)\leq c\}
  =
  \{r:\mathcal S(\mathcal H(r))\geq d\}
\]
for some $d>0$.  This set is convex because
$\mathcal S=\sigma_k^{1/k}$ is concave on $\Gamma_k^+$ and
$\mathcal H$ is affine in $r$.  To match the sign convention in the
parabolic regularity theorem, set $v=-u$.  Its equation is
\[
  v_\tau
  =\widetilde G(\theta,\tau,v,Dv,D^2v),
  \qquad
  \widetilde G(\theta,\tau,z,p,r)
  :=-G(\theta,\tau,-z,-p,-r).
\]
The operator $\widetilde G$ is still uniformly parabolic, and its
superlevel sets in the Hessian variable are convex.  Equivalently, if
$\widetilde G^{ij}M_{ij}=0$, then
\[
  \widetilde G^{ij,kl}M_{ij}M_{kl}\leq0.
\]
Indeed, on a tangent direction the first variation of $\mathcal S$
vanishes, so the term produced by differentiating the power
$\mathcal S^{k\alpha}$ twice vanishes, while the remaining term is
nonpositive by the concavity of $\mathcal S$.  Hence the classical
$C^{2,\vartheta}$ regularity theory for uniformly parabolic equations
with convex level sets applies; see
Andrews~\cite[Theorem~6]{Andrews04} and Krylov~\cite{Krylov87}.
This is the same regularity step used in~\cite{LXZ}.  Standard
parabolic Schauder estimates and bootstrapping then give all higher
derivative estimates on every compact normalized time interval.

If $\tau_{\max}<\infty$, the normalized $C^0$ estimate and the
boundedness of $\lambda$ on $[0,\tau_{\max})$ keep
$\rho=u/\lambda$ in a compact subinterval of $(0,\infty)$.  All
derivatives of the coefficients are therefore bounded up to
$\tau_{\max}$, and the preceding interior estimates, together with
the smooth short-time estimates near $\tau=0$, allow the solution to
be continued past $\tau_{\max}$.  This contradiction proves
$\tau_{\max}=\infty$.

Finally, under \eqref{eq:profile-symbol-bounds-all-orders},
\[
  \partial_z^j\!\left[\lambda^\beta f(z/\lambda)\right]
  =\lambda^{\beta-j}f^{(j)}(z/\lambda)
\]
is uniformly bounded for every $j\geq0$ when $z$ ranges in the compact
interval supplied by the normalized $C^0$ estimate.  Moreover, with
$r=z/\lambda$ and $z$ held fixed,
\[
  \partial_\tau\partial_z^j
  \!\left[\lambda^\beta f(z/\lambda)\right]
  =\gamma_0\lambda^{\beta-j}
    \left((\beta-j)f^{(j)}(r)-r f^{(j+1)}(r)\right).
\]
Thus \eqref{eq:lambda-f-tau-derivative} and the symbol bounds show
that every $\tau$-derivative of the coefficients is again a uniformly
bounded expression of the same form.  The normalized
hyperbolic factors and all their derivatives are uniformly bounded as
well.  Thus the constants in the $C^{2,\vartheta}$ and Schauder
estimates are uniform on moving time strips.  Combining these estimates
with the smooth short-time solution proves
\eqref{eq:all-higher-order-estimates-hyperbolic}.
\end{proof}

\begin{remark}
The all-orders condition
\eqref{eq:profile-symbol-bounds-all-orders} is used only to make the
higher-order estimates uniform as $\tau\to\infty$.  On every finite
normalized time interval, the smoothness of $f$ on $(0,\infty)$ is
sufficient for continuation.
\end{remark}

\section{Exponential Decay}
In this section, we prove the exponential decay of all positive-order
spatial derivatives of the normalized radial function. We first
give a lemma concerning only the radial weight function $f$ which will be useful in the gradient estimate. We then introduce a new
radial variable and apply the maximum principle to its gradient to
obtain exponential first-derivative decay.  Finally, the uniform
$C^\infty$ estimates from the previous section and the interpolation
inequalities on $\mathbb S^n$ propagate this decay to all spatial
derivatives.

\begin{lemma}
\label{lem:radial-weight-derivative-asymptotics-hyperbolic}
Assume that the radial weight satisfies the hypotheses of either
Theorem~\ref{thm:main} or Theorem~\ref{thm:main-critical}.  Then, as
$r\downarrow0$,
\begin{equation}\label{eq:profile-first-derivative-sign-hyperbolic}
  f'(r)-(1+k\alpha)f(r)\coth r
  =\begin{cases}
    (\beta-1-k\alpha+o(1))r^{\beta-1},
      &\beta>1+k\alpha,\\
    O(r^{k\alpha+\delta_0}),
      &\beta=1+k\alpha,
  \end{cases}
\end{equation}
where, in the critical case, $\delta_0:=\min\{\delta,2\}>0$.
\end{lemma}

\begin{proof}
Recall $f(r)=\sinh^\beta r+g(r)$. 
\begin{itemize}
  \item Under the assumptions of Theorem~\ref{thm:main}, we have
      \[
  g(r)=O(r^{m+1}),
  \qquad
  g'(r)=O(r^m),
\]
where $m = \lfloor\beta\rfloor$. Since $m>\beta-1$ and $\coth r=O(r^{-1})$, it follows that
\[
  g'(r)-(1+k\alpha)g(r)\coth r=o(r^{\beta-1}).
\]
On the other hand,
\[
  (\sinh^\beta r)'-(1+k\alpha)\sinh^\beta r\coth r
  =(\beta-1-k\alpha)\sinh^{\beta-1}r\cosh r.
\]
This proves the first line of
\eqref{eq:profile-first-derivative-sign-hyperbolic}.
  \item Under the assumptions of Theorem~\ref{thm:main-critical}, let
      $N:=\lceil\beta+\delta\rceil$ be the regularity order specified
      in \eqref{eq:C0-critical-regularity-order-hyperbolic}.  Since
      \[
        N-1<\beta+\delta\leq N,
      \]
      the assumptions $g\in C^N$ and $g(r)=O(r^{\beta+\delta})$ force
      \[
        g^{(j)}(0)=0,
        \qquad 0\leq j<N.
      \]
      Taylor's theorem applied to $g'$ then gives
      \[
        g'(r)=O(r^{N-1})=O(r^{\beta+\delta-1}),
      \]
      because $N-1\geq\beta+\delta-1$.  Since $\coth r=O(r^{-1})$, we
      also have
      \[
        g(r)\coth r=O(r^{\beta+\delta-1}).
      \]

Consequently,
\[
  f'(r)-(1+k\alpha)f(r)\coth r
  =O(r^{k\alpha+\delta})
  =O(r^{k\alpha+\delta_0}),
\]
which proves the second line.
\end{itemize}
\end{proof}

\begin{lemma}[Exponential decay of the normalized gradient]
\label{lem:exponential-gradient-decay-hyperbolic}
Assume that the conditions of either Theorem~\ref{thm:main} or
Theorem~\ref{thm:main-critical} hold.  Then there exist constants
$C,c>0$ such that
\begin{equation}\label{eq:exponential-gradient-decay-hyperbolic}
  |\bar{\nabla}\widetilde\rho(\cdot,\tau)|_\sigma
  \leq Ce^{-c\tau}
  \qquad\text{for all }\tau\geq0.
\end{equation}
\end{lemma}

\begin{proof} 
Define
\begin{equation}\label{eq:hyperbolic-logarithmic-radius}
  \varphi(\theta,\tau)
  :=\log\tanh\frac{\rho(\theta,\tau)}2.
\end{equation}
Since $\varphi_\rho=(\sinh\rho)^{-1}$, it follows
\begin{equation}\label{eq:varphi-gradient-relations}
  \bar{\nabla}\varphi=\frac{\bar{\nabla}\rho}{\sinh\rho},
  \qquad
  v=\sqrt{1+|\bar{\nabla}\varphi|_\sigma^2},
  \qquad
  \bar{\nabla}\widetilde\rho(\theta,\tau)
  =\lambda(\tau)\sinh\rho(\theta,\tau)\,\bar{\nabla}\varphi(\theta,\tau),
\end{equation}
and
\[
  \rho_i=\sinh\rho\,\varphi_i,
  \qquad
  \rho_{ij}
  =\sinh\rho\,\varphi_{ij}
   +\sinh\rho\cosh\rho\,\varphi_i\varphi_j.
\]
By Corollary~\ref{cor:lambda-sin-rho-two-sided-hyperbolic},
$\lambda\sinh\rho$ is uniformly bounded.  Hence the last identity in
\eqref{eq:varphi-gradient-relations} shows that it suffices to prove
the exponential decay of $|\bar{\nabla}\varphi|_\sigma$.
Substituting these identities into~\eqref{eq:weingarten-map-rho} and using~\eqref{eq:normalized-weingarten-map-definition}, we obtain
\begin{equation}\label{eq:hat-h-in-varphi-hyperbolic}
  \hat h_i{}^j
  =\frac{1}{\lambda v\sinh\rho}
    \left[
      \cosh\rho\,\delta_i{}^j
      -\left(
         \sigma^{jp}-\frac{\varphi^j\varphi^p}{v^2}
       \right)\varphi_{pi}
    \right].
\end{equation}
Replacing $\rho$ in \eqref{eq:radial-flow} by
$\varphi$ and using \eqref{eq:time-and-lambda-identities}, we obtain
\begin{equation}\label{eq:varphi-evolution-hyperbolic}
  \partial_\tau\varphi
  =-\lambda^{\beta-1}\frac{f(\rho)}{\sinh\rho}
      \sigma_k^\alpha(\hat h)v.
\end{equation}

Put
\begin{equation}\label{eq:gradient-auxiliary-function-hyperbolic}
  W:=\frac12|\bar{\nabla}\varphi|_\sigma^2,
\end{equation}
where the norm and derivatives are taken with respect to the standard
metric on $\mathbb S^n$. At a spatial maximum $(\theta_0,\tau)$ of $W$, we have
\begin{equation}\label{eq:C1-maximum-point-relation}
  \varphi^i\varphi_{ij}=0,\quad
  \bar{\nabla}v=0,\quad
  \text{ and }
  \bar{\nabla}^2W\leq0.
\end{equation}
Then equation \eqref{eq:varphi-evolution-hyperbolic} gives
\begin{equation}\label{eq:W-time-derivative-first-expansion-hyperbolic}
  \begin{aligned}
  \partial_\tau W
  ={}&-\varphi^m\bar{\nabla}_m
      \left[
        \lambda^{\beta-1}\frac{f(\rho)}{\sinh\rho}
        \sigma_k^\alpha(\hat h)v
      \right]\\
  ={}&-v\sigma_k^\alpha(\hat h)\,\varphi^m\bar{\nabla}_m
      \left(\lambda^{\beta-1}\frac{f(\rho)}{\sinh\rho}\right)-\lambda^{\beta-1}\frac{f(\rho)}{\sinh\rho}v
       \frac{\partial\sigma_k^\alpha}
            {\partial\hat h_i{}^j}
       \varphi^m\bar{\nabla}_m\hat h_i{}^j
         \end{aligned}
\end{equation}
where we used the fact that $\bar{\nabla}v=0$ at the maximum point.

Since $\rho_m=\sinh\rho\,\varphi_m$, we have
\begin{equation}\label{eq:coefficient-and-v-gradient-at-W-maximum-hyperbolic}
  \varphi^m\bar{\nabla}_m
  \left(\lambda^{\beta-1}\frac{f(\rho)}{\sinh\rho}\right)
  =\lambda^{\beta-1}\bigl(f'(\rho)-f(\rho)\coth\rho\bigr)
     |\bar{\nabla}\varphi|_\sigma^2.
\end{equation}

Differentiating \eqref{eq:hat-h-in-varphi-hyperbolic} and using~\eqref{eq:C1-maximum-point-relation}, we obtain 
\begin{equation}\label{eq:directional-derivative-hat-h-hyperbolic}
  \begin{aligned}
  \varphi^m\bar{\nabla}_m\hat h_i{}^j
  ={}&-\cosh\rho\,|\bar{\nabla}\varphi|_\sigma^2\hat h_i{}^j
      +\frac{\sinh\rho}{\lambda v}
       |\bar{\nabla}\varphi|_\sigma^2\delta_i{}^j-\frac{1}{\lambda v\sinh\rho}
       \left(
         \sigma^{jp}-\frac{\varphi^j\varphi^p}{v^2}
       \right)\varphi^m\varphi_{pi m}.
  \end{aligned}
\end{equation}
Using the $k\alpha$-homogeneity identity
\begin{equation}\label{eq:sigma-k-alpha-Euler-identity-C1-hyperbolic}
  \frac{\partial\sigma_k^\alpha}{\partial\hat h_i{}^j}
  \hat h_i{}^j
  =k\alpha\sigma_k^\alpha(\hat h),
\end{equation}
and substituting
\eqref{eq:coefficient-and-v-gradient-at-W-maximum-hyperbolic} and
\eqref{eq:directional-derivative-hat-h-hyperbolic} into
\eqref{eq:W-time-derivative-first-expansion-hyperbolic}, we obtain
\begin{equation}\label{eq:W-before-third-derivative-commutation-hyperbolic}
  \begin{aligned}
  \partial_\tau W
  ={}&\lambda^{\beta-2}f(\rho)
      (\dot\sigma_k^\alpha)^{pi}\varphi^m\varphi_{pi m}
   -\lambda^{\beta-1}v\sigma_k^\alpha(\hat h)
      \left[f'(\rho)-(1+k\alpha)f(\rho)\coth\rho\right]
      |\bar{\nabla}\varphi|_\sigma^2\\
   &-\lambda^{\beta-2}f(\rho)
      \frac{\partial\sigma_k^\alpha}
           {\partial\hat h_i{}^j}\delta_i{}^j
      |\bar{\nabla}\varphi|_\sigma^2.
  \end{aligned}
\end{equation}
where
\begin{equation}\label{eq:C1-linearized-coefficient-hyperbolic}
  (\dot\sigma_k^\alpha)^{pq}
  :=\frac{\partial\sigma_k^\alpha}{\partial\hat h_q{}^j}g^{jp} = \frac{\partial\sigma_k^\alpha}{\partial\hat h_q{}^j}
  \frac1{\sinh^2\rho}
  \left(
    \sigma^{jp}-\frac{\varphi^j\varphi^p}{v^2}
  \right).
\end{equation}
Clearly, $(\dot\sigma_k^\alpha)^{pq}$ is positive definite.

By the Ricci identity on $\mathbb{S}^n$, we have
\begin{equation}\label{eq:third-derivative-commutation-for-W-hyperbolic}
  \varphi^m\varphi_{pi m}
  =\bar{\nabla}_p\bar{\nabla}_iW
   -\varphi_p{}^m\varphi_{mi}
   -\bigl(\sigma_{pi}|\bar{\nabla}\varphi|_\sigma^2
           -\varphi_p\varphi_i\bigr).
\end{equation}
Consequently, at a maximum point of $W$,
\begin{equation}\label{eq:gradient-maximum-calculation-hyperbolic}
  \begin{aligned}
  \partial_\tau W
  ={}&\lambda^{\beta-2}f(\rho)
      (\dot\sigma_k^\alpha)^{pi}
      \bar{\nabla}_p\bar{\nabla}_iW
      -\lambda^{\beta-2}f(\rho)
      (\dot\sigma_k^\alpha)^{pi}
      \varphi_p{}^m\varphi_{mi}\\
   &-\lambda^{\beta-2}f(\rho)
      (\dot\sigma_k^\alpha)^{pi}
      \bigl(\sigma_{pi}|\bar{\nabla}\varphi|_\sigma^2
            -\varphi_p\varphi_i\bigr)\\
   &-\lambda^{\beta-1}v\sigma_k^\alpha(\hat h)
      \left[f'(\rho)-(1+k\alpha)f(\rho)\coth\rho\right]
      |\bar{\nabla}\varphi|_\sigma^2\\
   &-\lambda^{\beta-2}f(\rho)
      \frac{\partial\sigma_k^\alpha}
           {\partial\hat h_i{}^j}\delta_i{}^j
      |\bar{\nabla}\varphi|_\sigma^2.
  \end{aligned}
\end{equation}
If $W_{\max}(\tau)=0$, then
$\bar{\nabla}\varphi(\cdot,\tau)\equiv0$, so every point is a
maximizer of $W(\cdot,\tau)$ and
$\partial_\tau W=\varphi^m\bar{\nabla}_m(\partial_\tau\varphi)=0$
there; the desired differential inequality
\eqref{eq:gradient-maximum-component-expansion-hyperbolic} below is
then immediate.  Suppose now that $W_{\max}(\tau)>0$.
Choose at the maximum point an
orthonormal frame such that
\begin{equation}\label{eq:C1-adapted-frame-hyperbolic}
  \varphi_1=|\bar{\nabla}\varphi|_\sigma=\sqrt{2W},
  \qquad \varphi_a=0\quad(a=2,\ldots,n).
\end{equation}
Since $\varphi_1=\sqrt{2W_{\max}}>0$, the relation
$\varphi^i\varphi_{ij}=0$ at the maximum
point implies $\varphi_{1j}=0$ for every $j$.  We may therefore rotate
only the orthogonal complement of $\partial_{x^1}$ so that $(\varphi_{ab})_{a,b\geq2}$ is diagonal. Moreover,
\[
  \left(
    \sigma^{jp}-\frac{\varphi^j\varphi^p}{v^2}
  \right)
  =\operatorname{diag}(v^{-2},1,\ldots,1).
\]
It now follows from \eqref{eq:hat-h-in-varphi-hyperbolic} that
$\hat h_i{}^j$ is diagonal in the same frame.  Thus
$\partial_{x^1}$ is automatically a principal direction, and
\begin{equation}\label{eq:principal-curvatures-at-W-maximum-hyperbolic}
  \hat\kappa_1=\frac{\coth\rho}{\lambda v},
  \qquad
  \hat\kappa_a
  =\frac{\cosh\rho-\varphi_{aa}}
         {\lambda v\sinh\rho},
  \quad a=2,\ldots,n.
\end{equation}
We emphasize that, throughout the present argument, the index $1$
labels the gradient direction $\partial_{x^1}$ chosen in
\eqref{eq:C1-adapted-frame-hyperbolic}, so that
\eqref{eq:principal-curvatures-at-W-maximum-hyperbolic} and the
estimates below use direction-based labels; in particular,
$\hat\kappa_1$ here need not be the smallest principal curvature,
unlike the size-ordered convention
$\hat\kappa_1\leq\cdots\leq\hat\kappa_n$ used earlier (see the
paragraph preceding
\eqref{eq:normalized-profile-coefficient-bounds}).
Then,
$$\sigma_{pi}(\dot{\sigma}_k^{\alpha})^{pi} = \alpha\sigma_{k}^{\alpha-1}\frac{\p \sigma_k}{\p \hat{h}_i{ }^j}\delta_j{ }^i\frac{1}{\sinh^2\rho} - \alpha\sigma_k^{\alpha-1}\frac{\p\sigma_k}{\p \hat{h}_1{ }^1}\frac{1}{\sinh^2\rho}\frac{2W}{v^2}$$
and
$$(\dot{\sigma}_k^{\alpha})^{pi}\varphi_p\varphi_i = \alpha\sigma_{k}^{\alpha-1}\frac{\p \sigma_k}{\p \hat{h}_1{ }^1}\frac{1}{\sinh^2\rho}2W - \alpha\sigma_k^{\alpha-1}\frac{\p\sigma_k}{\p \hat{h}_1{ }^1}\frac{1}{\sinh^2\rho}\frac{4W^2}{v^2}.$$
Combining these two identities and using
$|\bar{\nabla}\varphi|_\sigma^2=2W$, we obtain
\begin{equation}\label{eq:hyperbolic-curvature-gradient-contraction}
  \begin{aligned}
  &\lambda^{\beta-2}f(\rho)
   (\dot\sigma_k^\alpha)^{pi}
   \bigl(\sigma_{pi}|\bar{\nabla}\varphi|_\sigma^2
         -\varphi_p\varphi_i\bigr)\\
  ={}&\frac{2W\alpha\lambda^{\beta-2}f(\rho)
             \sigma_k^{\alpha-1}}{\sinh^2\rho}
       \left(
         \frac{\partial\sigma_k}{\partial\hat h_i{}^j}
           \delta_i{}^j
         -\frac{\partial\sigma_k}{\partial\hat h_1{}^1}
       \right)\\
  ={}&\frac{2W\alpha\lambda^{\beta-2}f(\rho)
             \sigma_k^{\alpha-1}}{\sinh^2\rho}
       \sum_{a=2}^n
         \frac{\partial\sigma_k}{\partial\hat\kappa_a}.
  \end{aligned}
\end{equation}
Consequently, the first term in
\eqref{eq:gradient-maximum-calculation-hyperbolic} is nonpositive.
Writing
$W_{\max}(\tau):=\max_{\mathbb S^n}W(\cdot,\tau)$, the upper Dini
derivative of $W_{\max}$ satisfies
\[
  D^+W_{\max}(\tau)
  \leq\max_{\theta\in\operatorname{Argmax}W(\cdot,\tau)}
      \partial_\tau W(\theta,\tau)
\]
(see the convention on spatial extrema in the Notation section).
Since the preceding calculation holds at every maximizer of
$W(\cdot,\tau)$, the remaining terms give
\begin{equation}\label{eq:gradient-maximum-component-expansion-hyperbolic}
  \begin{aligned}
  D^+W_{\max}
  \leq{}&-\lambda^{\beta-1}v\sigma_k^\alpha(\hat h)
      \left[f'(\rho)-(1+k\alpha)f(\rho)\coth\rho\right]
      2W_{\max}\\
   &-\lambda^{\beta-2}\frac{f(\rho)}{\sinh^2\rho}
      \sum_{a=2}^n
      \frac{\partial\sigma_k^\alpha}{\partial\hat\kappa_a}
      \varphi_{aa}^2\\
   &-\lambda^{\beta-2}f(\rho)2W_{\max}
      \left(
        \frac{\partial\sigma_k^\alpha}{\partial\hat\kappa_1}
        +\coth^2\rho\sum_{a=2}^n
         \frac{\partial\sigma_k^\alpha}{\partial\hat\kappa_a}
      \right).
  \end{aligned}
\end{equation}

We now use the normalized $C^2$ estimate.  By
Lemma~\ref{lem:normalized-C2-estimate-hyperbolic}, the principal
curvature vector $\hat\kappa$ remains in a fixed compact subset of the
positive cone.  Hence
\begin{equation}\label{eq:uniform-ellipticity-after-C2-hyperbolic}
  0<c\leq
  \frac{\partial\sigma_k^\alpha}{\partial\hat\kappa_i}
  \leq C,
  \qquad
  0<c\leq\sigma_k^\alpha(\hat h)\leq C.
\end{equation}
Moreover, the $C^0$ estimate and
\eqref{eq:normalized-profile-coefficient-bounds} give
\begin{equation}\label{eq:decay-coefficients-hyperbolic}
  0<c\leq\lambda^\beta f(\rho)\leq C,
  \qquad
  0<c\leq\lambda^{-1}\coth\rho\leq C.
\end{equation}

By Lemma~\ref{lem:radial-weight-derivative-asymptotics-hyperbolic},
the radial-weight contribution in
\eqref{eq:gradient-maximum-component-expansion-hyperbolic} is
nonpositive for all sufficiently large $\tau$ in the supercritical
case.  In the critical case, since $\beta-1=k\alpha$ and
$\rho=\widetilde\rho/\lambda$, the same lemma and the normalized
$C^0$ estimate give
\[
  \lambda^{\beta-1}
  \left|f'(\rho)-(1+k\alpha)f(\rho)\coth\rho\right|
  \leq C(\lambda\rho)^{k\alpha}\rho^{\delta_0}
  \leq Ce^{-\delta_0\gamma_0\tau}.
\]
Together with the normalized $C^1$ estimate and
\eqref{eq:uniform-ellipticity-after-C2-hyperbolic}, this shows that the
radial-weight contribution is bounded above by
$Ce^{-\delta_0\gamma_0\tau}W_{\max}$.

Since $n\geq2$, using
\eqref{eq:uniform-ellipticity-after-C2-hyperbolic} and
\eqref{eq:decay-coefficients-hyperbolic}, the last line of
\eqref{eq:gradient-maximum-component-expansion-hyperbolic} satisfies
\begin{equation}\label{eq:hyperbolic-ambient-gradient-coercivity}
  \begin{aligned}
  &-2W_{\max}\lambda^{\beta-2}f(\rho)
      \left(
        \frac{\partial\sigma_k^\alpha}{\partial\hat\kappa_1}
        +\coth^2\rho\sum_{a=2}^n
         \frac{\partial\sigma_k^\alpha}{\partial\hat\kappa_a}
      \right)\\
  \leq{}&-2W_{\max}\lambda^\beta f(\rho)
      (\lambda^{-1}\coth\rho)^2
          \sum_{a=2}^n
          \frac{\partial\sigma_k^\alpha}{\partial\hat\kappa_a}
      \\
  \leq{}&-cW_{\max}
  \end{aligned}
\end{equation}
for all sufficiently large $\tau$.  The Hessian-square term is
nonpositive, and the exponentially small radial-weight error in the
critical case can be absorbed into the right-hand side above.  We
consequently obtain
\begin{equation}\label{eq:W-exponential-inequality-n-geq-two}
  D^+W_{\max}\leq-cW_{\max}
\end{equation}
for all sufficiently large $\tau$.

Integrating this Dini inequality and enlarging the constant
to cover the remaining compact time interval yields
\begin{equation}\label{eq:W-exponential-decay-hyperbolic}
  W_{\max}(\tau)\leq Ce^{-c\tau}.
\end{equation}
Finally, \eqref{eq:varphi-gradient-relations} and the upper bound for
$\lambda\sinh\rho$ imply
\[
  |\bar{\nabla}\widetilde\rho|_\sigma
  =\lambda\sinh\rho\,|\bar{\nabla}\varphi|_\sigma
  \leq C\sqrt{W_{\max}}
  \leq Ce^{-c\tau},
\]
after changing $c$.  This proves
\eqref{eq:exponential-gradient-decay-hyperbolic}.
\end{proof}

\begin{corollary}[Exponential decay of all spatial derivatives]
\label{cor:all-spatial-derivatives-exponential-decay-hyperbolic}
Assume that the conditions of either Theorem~\ref{thm:main} or
Theorem~\ref{thm:main-critical} hold.  Then, for every integer
$m\geq1$, there exist constants $C_m,c_m>0$ such that
\begin{equation}\label{eq:all-spatial-derivatives-exponential-decay}
  \|\bar{\nabla}^{m}\widetilde\rho(\cdot,\tau)\|_{C^0}
  \leq C_me^{-c_m\tau}
  \qquad\text{for all }\tau\geq0.
\end{equation}
\end{corollary}

\begin{proof}
The case $m=1$ is Lemma~\ref{lem:exponential-gradient-decay-hyperbolic}.
We use the interpolation argument of
Li--Sheng--Wang~\cite[(5.1)]{LSWJEMS}.  For every smooth tensor field
$T$ on $\mathbb S^n$ and integers $0\leq j\leq N$, the interpolation
inequality reads
\begin{equation}\label{eq:spherical-L2-interpolation-inequality}
  \int_{\mathbb S^n}|\bar{\nabla}^{j}T|_\sigma^2\,d\mu_\sigma
  \leq C_{N,n}
  \left(\int_{\mathbb S^n}|\bar{\nabla}^{N}T|_\sigma^2
    \,d\mu_\sigma\right)^{j/N}
  \left(\int_{\mathbb S^n}|T|_\sigma^2\,d\mu_\sigma
  \right)^{1-j/N}.
\end{equation}
Fix $m\geq2$, choose an integer $\ell>n/2$, and then choose
$N>m+\ell$.  Applying
\eqref{eq:spherical-L2-interpolation-inequality} to
$T=\bar{\nabla}\widetilde\rho$, with $j=m+s-1$ and with $N$ there
replaced by $N-1$, gives, for $0\leq s\leq\ell$,
\[
  \begin{aligned}
  \int_{\mathbb S^n}
    |\bar{\nabla}^{m+s}\widetilde\rho|_\sigma^2\,d\mu_\sigma
  \leq{}& C
  \left(\int_{\mathbb S^n}
    |\bar{\nabla}^{N}\widetilde\rho|_\sigma^2\,d\mu_\sigma
  \right)^{\frac{m+s-1}{N-1}}
  \left(\int_{\mathbb S^n}
    |\bar{\nabla}\widetilde\rho|_\sigma^2\,d\mu_\sigma
  \right)^{1-\frac{m+s-1}{N-1}}.
  \end{aligned}
\]
The first factor on the right is uniformly bounded by
Proposition~\ref{prop:C2-to-Cinfty-normalized-hyperbolic}, while the
second decays exponentially by
Lemma~\ref{lem:exponential-gradient-decay-hyperbolic}.  Consequently,
for $0\leq s\leq\ell$,
\[
  \int_{\mathbb S^n}
    |\bar{\nabla}^{m+s}\widetilde\rho|_\sigma^2\,d\mu_\sigma
  \leq C_m e^{-2c_m\tau}
\]
after decreasing $c_m>0$ if necessary.  The Sobolev embedding theorem
on $\mathbb S^n$ now yields
\[
  \|\bar{\nabla}^{m}\widetilde\rho(\cdot,\tau)\|_{C^0}
  \leq C
  \left(
    \sum_{s=0}^{\ell}
    \int_{\mathbb S^n}
      |\bar{\nabla}^{m+s}\widetilde\rho|_\sigma^2\,d\mu_\sigma
  \right)^{1/2}
  \leq C_me^{-c_m\tau}.
\]
This proves
\eqref{eq:all-spatial-derivatives-exponential-decay}.
\end{proof}

\section{Proof of Theorems}

\begin{proof}[Proof of Theorems~\ref{thm:main} and
\ref{thm:main-critical}]
Proposition~\ref{prop:C2-to-Cinfty-normalized-hyperbolic} applies after
the normalized $C^2$ estimate.  Its finite-strip estimates and the
continuation argument give long-time existence.  Under
\eqref{eq:profile-symbol-bounds-all-orders}, its moving-strip estimates
give \eqref{eq:all-higher-order-estimates-hyperbolic} uniformly in
$\tau$.

By
Corollary~\ref{cor:all-spatial-derivatives-exponential-decay-hyperbolic},
all positive-order spatial derivatives of $\widetilde\rho$ decay
exponentially as in
\eqref{eq:all-spatial-derivatives-exponential-decay}.
In particular, the oscillation of $\widetilde\rho$ decays
exponentially.  Thus the only possible limit is spatially constant.

For completeness, set
\[
  u(\theta,\tau):=\widetilde\rho(\theta,\tau),
  \qquad
  \overline\rho(\tau)
  :=\frac{1}{|\mathbb S^n|}
    \int_{\mathbb S^n}u(\theta,\tau)\,d\mu_\sigma.
\]
The exponential decay of all nonconstant spatial modes gives
\begin{equation}\label{eq:constant-mode-spatial-expansion}
  u=\overline\rho+O(e^{-c\tau}),
  \qquad
  |\bar{\nabla}u|_\sigma
   +|\nabla_\sigma^2u|_\sigma
  =O(e^{-c\tau}),
\end{equation}
uniformly on $\mathbb S^n$.  Since $u$ has fixed positive upper and
lower bounds, the radial graph formula, together with
\[
  \lambda^{-1}\coth\frac{u}{\lambda}
  =u^{-1}+O(\lambda^{-2}),
\]
implies
\begin{equation}\label{eq:constant-mode-curvature-expansion}
  v=1+O(e^{-c\tau}),
  \qquad
  \hat h_i{}^j
  =u^{-1}\delta_i{}^j
    +O(e^{-c\tau})+O(\lambda^{-2}).
\end{equation}
Here and below all tensorial error terms are measured with respect to
$\sigma$.  Since $\hat h$ stays in a fixed compact subset of the
positive cone, the smoothness of $\sigma_k^\alpha$ and the definition
of $\gamma_0$ yield
\begin{equation}\label{eq:constant-mode-curvature-function-expansion}
  \sigma_k^\alpha(\hat h)
  =\gamma_0u^{-k\alpha}
   +O(e^{-c\tau})+O(\lambda^{-2}).
\end{equation}

We next expand the radial weight.  In the supercritical case, with
$m=\lfloor\beta\rfloor$, the flatness assumption gives
$g(r)=O(r^{m+1})$.  Hence
\begin{equation}\label{eq:supercritical-normalized-weight-expansion}
  \lambda^\beta f(u/\lambda)
  =u^\beta
   +O(\lambda^{-2})
   +O(\lambda^{-(m+1-\beta)})
  =u^\beta+O(e^{-c\tau}),
\end{equation}
because $m+1-\beta>0$.  In the critical case,
$g(r)=O(r^{\beta+\delta})$ gives instead
\begin{equation}\label{eq:critical-normalized-weight-expansion}
  \lambda^\beta f(u/\lambda)
  =u^\beta
   +O(\lambda^{-2})+O(\lambda^{-\delta})
  =u^\beta+O(e^{-c\tau}).
\end{equation}
Substituting
\eqref{eq:constant-mode-curvature-expansion}--
\eqref{eq:critical-normalized-weight-expansion} into
\eqref{eq:normalized-radial-flow-hat-h}, and decreasing $c>0$ if
necessary, we obtain the pointwise equation
\begin{equation}\label{eq:constant-mode-pointwise-asymptotic-ode}
  u_\tau
  =\gamma_0u-\gamma_0u^{\beta-k\alpha}
   +O(e^{-c\tau}).
\end{equation}
After averaging, \eqref{eq:constant-mode-spatial-expansion} and the
two-sided positive bounds for $u$ give
\begin{equation}\label{eq:constant-mode-average-asymptotic-ode}
  \overline\rho'
  =\gamma_0\overline\rho
   -\gamma_0\overline\rho^{\,\beta-k\alpha}
   +O(e^{-c\tau}).
\end{equation}

In the supercritical case, put
$\mu:=\beta-k\alpha-1>0$.  Equation
\eqref{eq:constant-mode-average-asymptotic-ode} becomes
\begin{equation}\label{eq:supercritical-average-asymptotic-ode}
  \overline\rho'
  =\gamma_0\overline\rho
     \bigl(1-\overline\rho^\mu\bigr)
   +O(e^{-c\tau}).
\end{equation}
Set $z:=\overline\rho^{-\mu}$.  Since $\overline\rho$ is bounded above
and away from zero, differentiating
\eqref{eq:supercritical-average-asymptotic-ode} gives
\[
  z'=-\mu\gamma_0(z-1)+O(e^{-c\tau}).
\]
The variation-of-constants formula therefore yields
\[
  |z(\tau)-1|\leq Ce^{-c'\tau},
  \qquad
  c':=\frac12\min\{c,\mu\gamma_0\}>0.
\]
Since $\overline\rho$ remains in a fixed compact subinterval of
$(0,\infty)$, it follows, after relabelling $c'>0$, that
\begin{equation}\label{eq:supercritical-average-convergence}
  |\overline\rho(\tau)-1|\leq Ce^{-c\tau}.
\end{equation}
Together with
\eqref{eq:all-spatial-derivatives-exponential-decay}, this proves that
$\widetilde\rho$ converges exponentially to $1$ in every $C^m$ norm.

In the critical case $\beta=1+k\alpha$, the two leading terms in
\eqref{eq:constant-mode-average-asymptotic-ode} cancel and give
\begin{equation}\label{eq:critical-average-integrable-derivative}
  |\overline\rho'(\tau)|\leq Ce^{-c\tau}.
\end{equation}
Consequently, there exists $R_\infty>0$ such that
\begin{equation}\label{eq:critical-average-convergence}
  |\overline\rho(\tau)-R_\infty|
  \leq Ce^{-c\tau}.
\end{equation}
Again using
\eqref{eq:all-spatial-derivatives-exponential-decay}, we conclude that
$\widetilde\rho$ converges exponentially to $R_\infty$ in every
$C^m$ norm.

Thus, in either case, the normalized radial graphs converge smoothly
and exponentially to a constant radial graph, namely a geodesic
sphere centred at $o$.  Since
$\rho=\widetilde\rho/\lambda$ and $\lambda\to\infty$, the original
hypersurfaces contract smoothly to $o$ as $t\to\infty$.  This completes
the proofs of both theorems.
\end{proof}

{\bf AI usage.}
The authors used GPT-5.6 Sol and KIMI K3 during the preparation of this manuscript. Specifically, the AI model was used to assist with exploratory calculations for candidate model functions, such as $f(\rho)=\rho^{\beta}$ and $f(\rho)=\coth^{\beta}\rho$. These preliminary computations suggested the need for additional convexity assumptions beyond the $h$-convexity of the initial hypersurface. The insights gained informed the authors'  subsequent analysis of such assumptions for the model function $f(\rho)=\sinh^{\beta}\rho$, which was ultimately adopted. The authors have thoroughly checked all mathematical derivations and proofs and take full responsibility for the entire content of this manuscript.

{\bf Acknowledgements.}
W. Sheng was partially supported by National Key R$\&$D Program of China (No. 2022YFA1005500) and Natural Science Foundation of China under Grant No. 12571063.


\end{document}